\documentclass{elsart}

\usepackage{amsfonts}
\usepackage{makeidx}
\usepackage{graphicx}
\usepackage[english,activeacute]{babel}
\usepackage[latin1]{inputenc}
\usepackage{amsmath,amssymb,latexsym,comment}
\usepackage[numbers,sort&compress]{natbib}
\usepackage{hyperref}
\usepackage{fancybox}

\bibpunct{[}{]}{,}{n}{}{;}
\newtheorem{theorem}{Theorem}[section]

\newtheorem{corollary}{Corollary}[section]

\newtheorem{lemma}{Lemma}[section]

\newenvironment{proof}[1][Proof]{\noindent\textbf{#1.} }{\ \rule{0.5em}{0.5em}}
\newcommand{\fracd}[2]{\displaystyle{\frac{{\displaystyle{#1}}}{{\displaystyle{#2}}}}}

\journal{J. }

\begin{document}

\begin{frontmatter}

%% Title, authors and addresses

%% use the tnoteref command within \title for footnotes;
%% use the tnotetext command for the associated footnote;
%% use the fnref command within \author or \address for footnotes;
%% use the fntext command for the associated footnote;
%% use the corref command within \author for corresponding author footnotes;
%% use the cortext command for the associated footnote;
%% use the ead command for the email address,
%% and the form \ead[url] for the home page:
%%
%% \title{Title\tnoteref{label1}}
%% \tnotetext[label1]{}
%% \author{Name\corref{cor1}\fnref{label2}}
%% \ead{email address}
%% \ead[url]{home page}
%% \fntext[label2]{}
%% \cortext[cor1]{}
%% \address{Address\fnref{label3}}
%% \fntext[label3]{}

\title{A single parameter Hermite-Pad\'e series representations for Ap\'ery's constant}

%% use optional labels to link authors explicitly to addresses:
%% \author[label1,label2]{<author name>}
%% \address[label1]{<address>}
%% \address[label2]{<address>}

\author{Anier Soria-Lorente$^{1}$ and Stefan Berres$^{2}$}

\address{$^1$ Department of Basic Sciences, Granma University\\
	Km 17.5 de la carretera de Bayamo-Manzanillo, Bayamo, Cuba\\
	asorial@udg.co.cu\\[2pt] $^2$ 
	Departamento de Ciencias Matem\'{a}ticas y F\'{\i}sicas,
	Facultad de Ingenier\'{i}a,  Universidad Cat\'olica de Temuco, Temuco, Chile \\
	sberres@uct.cl}

\begin{abstract}
Inspired by the results of Rhin and Viola (2001), the purpose of this work is to elaborate on a series representation for $\zeta \left( 3\right)$ which only depends on one single integer parameter. This is accomplished by deducing a Hermite-Pad\'e approximation problem 
using ideas of Sorokin (1998). As a consequence we get a new recurrence relation for the approximation of $\zeta(3)$  as well as a corresponding new continued fraction expansion for $\zeta(3)$, which do no reproduce Ap\'ery's phenomenon, i.e., though the approaches are different, they lead to the same sequence of diophantine approximations to $\zeta \left( 3\right) $. Finally, the convergence rates of several series representations of $\zeta(3)$ are compared.
\end{abstract}

\begin{keyword}
Riemann zeta function, Ap\'{e}ry's theorem, Hermite-Pad\'{e} approximation problem, recurrence relation, continued fraction expansion, series representation.

\MSC Primary 11B37, 30B70, 14G10, 11J72,
11M06; Secondary 37B20, 11A55, 11J70, 11Y55, 11Y65
%% or \MSC[2008] code \sep code (2000 is the default)

\end{keyword}

\end{frontmatter}

% \linenumbers

%% main text

\section{Introduction}
The study of the arithmetical properties of the Riemann zeta function at
integer arguments%
\begin{equation*}
\zeta \left( k\right) :=\sum_{n\geq 1}\fracd{1}{n^{k}}=\fracd{\left( -1\right)
	^{k-1}}{\left( k-1\right) !}\int_{0}^{1}\fracd{\log ^{k-1}x}{1-x}\, dx,\quad
k=1,2,\ldots ,
\end{equation*}%
as well as the results related to its series representations, has
fascinated quite a number of mathematicians since the first half of the XVII
century \cite{Soria1,Soria2}, both for its theoretical implications and
practical applications \cite{Kellerhals,Webster}. Indeed, everything began
when in 1644 the Italian mathematician Pietro Mengoli proposed the famous
Basel problem in mathematical analysis, which also has relevance to number theory.
Nine decades later, this problem was solved by Leonhard Euler. In his famous
book on Differential Calculus of 1755 he gave the general case \cite%
{Balanzario,Laga}%
\begin{equation*}
\zeta \left( 2k\right) =\left( -1\right) ^{k-1}\fracd{\left( 2\pi \right)
	^{2k}B_{2k}}{2\left( 2k\right) !},\quad k=1,2,\ldots ,
\end{equation*}%
which is Euler's celebrated formula, where $B_{2k}$ are the so-called
Bernoulli numbers \cite{Abramowitz,Conway}, with $B_{2k}\in \mathbb{Q}$ for
all $k\in \mathbb{N}$. The generalization of the so-called Basel problem by
Euler was a very important step in number theory. Later on % , the enthusiastic
Euler proposed the following conjecture for the odd case,
\begin{equation*}
\zeta \left( 2k+1\right) =\fracd{p_{k}}{q_{k}}\pi ^{2k+1},
\end{equation*}
where $p_{k}$ and $q_{k}$ are integer numbers. However, Euler's efforts to
validate his conjecture did not work out \cite{Soria1}, and meanwhile the
conjecture itself has been refuted \cite{Takaaki}.

Regardless of Euler's frustrated attempts, he was able to derive the following series
representation
\begin{equation}
\zeta \left( 3\right) =\lim_{n\rightarrow \infty }\zeta _{n}^{E}\left(
3\right) ,  \label{ESR1}
\end{equation}%
where%
\begin{equation}
\zeta _{n}^{E}\left( 3\right) =-\fracd{4\pi ^{2}}{7}\sum_{k=0}^{n}\fracd{\zeta
	\left( 2k\right) }{\left( 2k+1\right) \left( 2k+2\right) 2^{2k}} . 
\label{ESR}
\end{equation}%
This representation has inspired a large number of mathematicians and has been recently discovered by several authors in many different ways \cite{ChengS,Srivastava1,Srivastava2}.

After these first investigations by Euler, nothing was known on the
arithmetical nature of the Riemann zeta function for odd arguments, until,
on a thursday afternoon in June 1978, at 2 pm, at the Journ\'es Arithmetiques
held at Marseille-Luminy, Roger Ap\'ery surprised the mathematical community
with a talk about the irrationality of $\zeta \left( 3\right) $, see for
instance \cite{apery,Cohen,Soria1,van-der-poorten}. In this talk he claimed
to have proofs that both $\zeta \left( 2\right) $ and $\zeta \left( 3\right) 
$ were irrational.

The rational approximants of Ap\'ery $p_{n}/q_{n}$, which are also named 
% Ap\'ery's approximants or 
Ap\'ery's diophantine approximations, % are  $p_{n}/q_{n}$,
approach $\zeta (3)$ as $n$ increases, i.e. converge for sufficiently large~$%
n$ as%
\begin{equation*}
\lim_{n\rightarrow \infty}\Bigl|\zeta (3)-\displaystyle{\fracd{{\displaystyle{p_{n}}}}{{\displaystyle{\
				q_{n}}}}}\Bigr|=0.
\end{equation*}%
One of the most crucial ingredients in Ap\'ery's proof is the existence of the
recurrence relation \cite{Elsner,Soria1,van-der-poorten}%
\begin{equation}
(n+2)^{3}y_{n+2}-(2n+3)(17n^{2}+51n+39)y_{n+1}+\left( n+1\right)
^{3}y_{n}=0,\quad n\geq 0,  \label{Apery-equat}
\end{equation}%
which is satisfied simultaneously by both the numerators $p_{n}$ and
denominators $q_{n}$ of the diophantine approximations $p_{n}/q_{n}$ to $%
\zeta \left( 3\right) $ with the respective initial condition%
\begin{equation*}
p_{0}=0,\quad p_{1}=6,\quad q_{0}=1,\quad q_{1}=5.
\end{equation*}%
The rational approximants $p_{n}$ and $q_{n}$ are also given by the explicit
representation of the sequences in question \cite%
{apery,Cohen,van-der-poorten}%
\begin{equation}
q_{n}:=\sum_{0\leq k\leq n}\binom{n+k}{k}^{2}\binom{n}{k}^{2}\quad \mbox{and}%
\quad p_{n}:=\sum_{0\leq k\leq n}\binom{n+k}{k}^{2}\binom{n}{k}^{2}\gamma
_{n,k},  \label{RA_Ap}
\end{equation}%
where%
\begin{equation*}
\gamma _{n,k}=\sum_{1\leq j\leq n}\fracd{1}{j^{3}}+\sum_{1\leq j\leq k}\fracd{%
	(-1)^{j-1}}{2j^{3}}\binom{n+j}{j}^{-1}\binom{n}{j}^{-1}.
\end{equation*}%
Observe that, from $\left( \text{\ref{Apery-equat}}\right) $ we can deduce
that%
\begin{equation}
\det 
\begin{pmatrix}
p_{n} & q_{n} \\ 
p_{n-1} & q_{n-1}%
\end{pmatrix}%
=\fracd{6}{n^{3}},  \label{detAp}
\end{equation}%
see \cite{Cohen,van-der-poorten} for more details.

In order to reformulate the recurrence relations 
in terms of a continued fraction representation
let us recall its definition and a basic lemma. 
We say that a number $\alpha $ can be written as an infinite
irregular continued fraction expansion, 
if it admits the following representation 
\begin{equation*}
\alpha =a_{0}+\fracd{b_{1}\mid }{\mid a_{1}}+\fracd{b_{2}\mid }{\mid a_{2}}%
+\cdots +\fracd{b_{n}\mid }{\mid a_{n}}+\cdots =a_{0}+\displaystyle\fracd{b_{1}%
}{a_{1}+\displaystyle\fracd{b_{2}}{\displaystyle a_{2}+\displaystyle\fracd{%
			b_{3}}{\displaystyle a_{3}+_{\ddots _{\displaystyle a_{n-1}+\fracd{b_{n}}{%
						a_{n}+_{\ddots}}}}}}}
\end{equation*}

\begin{lemma} \label{lem:jones}
	\cite[p.~31]{Jones} Let $\left( p_{n}\right) _{n\geq -1}$ and $\left(
	q_{n}\right) _{n\geq -1}$ be two sequences of numbers such that $q_{-1}=0$, $%
	p_{-1}=q_{0}=1$ and $p_{n}q_{n-1}-p_{n-1}q_{n}\neq 0$ for $n=0,1,2,\ldots $.
	Then there exists a unique irregular continued fraction 
	\begin{equation}
	a_{0}+\fracd{b_{1}\mid }{\mid a_{1}}+\fracd{b_{2}\mid }{\mid a_{2}}+\fracd{%
		b_{3}\mid }{\mid a_{3}}+\cdots +\fracd{b_{n}\mid }{\mid a_{n}}+\cdots ,
	\label{ICF}
	\end{equation}%
	whose $n$-th numerator is $p_{n}$ and $n$-th denominator is $q_{n}$, for
	each $n\geq 0$. More precisely,
	\begin{equation*}
	a_{0}=p_{0},\quad a_{1}=q_{1},\quad b_{1}=p_{1}-p_{0}q_{1},
	\end{equation*}%
	\begin{equation*}
	a_{n}=\fracd{p_{n}q_{n-2}-p_{n-2}q_{n}}{p_{n-1}q_{n-2}-p_{n-2}q_{n-1}},\quad
	b_{n}=\fracd{p_{n-1}q_{n}-p_{n}q_{n-1}}{p_{n-1}q_{n-2}-p_{n-2}q_{n-1}},\quad
	n=0,1,2,\ldots
	\end{equation*}%
	\label{BorweinT1}
\end{lemma}
Using Lemma~\ref{lem:jones}, from $\left( \text{\ref{detAp}}\right) $ we obtain 
\begin{equation*}
\zeta \left( 3\right) =\fracd{6\mid }{\mid 5}-\fracd{\ \ 1\text{ }\mid }{\mid
	117}-\fracd{\text{ }64\mid }{\mid 535}-\cdots -\fracd{\quad \quad \quad \quad
	\quad \quad n^{6}\quad \quad \quad \quad \mid }{\mid \left( 2n+1\right)
	\left( 17n^{2}+17n+5\right) }-\cdots \quad .
\end{equation*}%
Then, recognizing that $\zeta \left( 3\right) -p_{0}/q_{0}=\zeta \left(
3\right) $, it can be induced that (see \cite{Cohen,van-der-poorten} for
more details)%
\begin{equation*}
\left\vert \zeta \left( 3\right) -\fracd{p_{n}}{q_{n}}\right\vert
=\sum_{k\geq n+1}\fracd{6}{k^{3}q_{k-1}q_{k}}=\mathcal{O}\left(
q_{n}^{-2}\right) .
\end{equation*}%
Observe that, by the recurrence relation $\left( \text{\ref{Apery-equat}}%
\right) $ and Poincar\'e's theorem \cite{perron,poincare} that $q_{n}=\mathcal{%
	O}\left( \varpi ^{4n}\right) $, where $\varpi =\sqrt{2}+1$ is the silver
ratio. Moreover, by the prime number theorem, it can be shown that%
\begin{equation}
\mathcal{L}_{n}:= \prod_{p\leq n}p^{\left[ \fracd{\log n}{\log p}\right]
}\leq \prod_{p\leq n}n=\mathcal{O}\left( e^{\left( 1+\epsilon \right)
	n}\right) ,\quad \forall \text{ }\epsilon >0,  \label{LCM_AB}
\end{equation}%
where the product is over the prime numbers $p$ below or equal to $n$.
Therefore, setting $v_{n}=2p_{n}\mathcal{L}_{n}^{3}\in \mathbb{Z}$ and $%
u_{n}=2q_{n}\mathcal{L}_{n}^{3}\in \mathbb{Z}$, we obtain $u_{n}=\mathcal{O}%
\left( \varpi ^{4n}e^{3n}\right) $ and%
\begin{equation*}
\left\vert \zeta \left( 3\right) -\fracd{v_{n}}{u_{n}}\right\vert =\mathcal{O}%
\left( u_{n}^{-\left( 1+\delta \right) }\right) ,
\end{equation*}%
where%
\begin{equation*}
\delta =\fracd{\log \alpha -3}{\log \alpha +3}=0.080529\ldots >0.
\end{equation*}%
This proves Ap\'ery's theorem by virtue of the criterion for irrationality. 

\begin{theorem}
	If there is a $\delta >0$ and a sequence $\left( v_{n}/u_{n}\right) _{n\geq
		0}$ of rational numbers such that $v_{n}/u_{n}\neq x$ and%
	\begin{equation*}
	\left\vert x-\frac{v_{n}}{u_{n}}\right\vert <\frac{1}{u_{n}^{1+\delta }}%
	,\quad n=0,1,\ldots ,
	\end{equation*}%
	then $x$\ is irrational.
\end{theorem}

Ap\'ery's irrationality proof of $\zeta \left( 3\right) $ operates with the
series representation%
\begin{equation*}
\zeta \left( 3\right) =\fracd{5}{2}\sum_{n\geq 1}\fracd{\left( -1\right) ^{n-1}%
}{n^{3}\binom{2n}{n}},
\end{equation*}%
which converge faster than $\left( \text{\ref{ESR1}}\right) $, see \cite%
{Srivastava1,Srivastava2} for more details. The same was first obtained by
A. A. Markov in 1890 \cite{Markoff}.
In addition, Ap\'ery's recurrence
relation $\left( \text{\ref{Apery-equat}}\right) $ leads to the
characteristic equation $\lambda ^{2}-34\lambda +1=0$, which is associated
to the irrationality measure $\mu =13$.$4178202\ldots $, see \cite{zudilin-0}
for more details.

Although initially somewhat controversial, the aforementioned result
inspired several mathematicians to construct different methods to explain
the irrationality of $\zeta \left( 3\right) $ \cite%
{Arvesu2,beuker-4,Nesterenko1,Prevost,sorokin-1,sorokin-2,van-assche-2} as
well as to obtain other results related with the aforementioned constant 
\cite{Alzer,Cheng,Janous,Kondratieva,Lima}. Ap\'ery's phenomenon consists of
the observation that some of these alternative methods leads to the same
sequence of diophantine approximations $\left( \text{\ref{RA_Ap}}\right) $
to $\zeta \left( 3\right) $, and therefore, to the same characteristic
equation $\lambda ^{2}-34\lambda +1=0$, corresponding to the recurrence
relation $\left( \text{\ref{Apery-equat}}\right) $, which is associated to
the irrationality measure \cite{zudilin-0} obtained from Ap\'ery's results 
\cite{apery,Arvesu2,Cohen}.

In Section \ref{sec:hist} variants of Ap\'ery's phenomenon are recalled in order to put the new contribution into context.
Though the irrationality of $\zeta(3)$ has been shown with different approaches, always the same rational approximants of Ap\'erys are obtained.

In Section \ref{SeccHP}, 
we follow the aforementioned approaches of Rhin and Viola \cite%
{RViola1,RViola},and present a series representation for $\zeta \left(
3\right) $, which only depends on one single integer parameter. 
As a modification of the approach of Nesterenko (1996)
we propose to replace \eqref{nest96} by \eqref{NNest}.
% This modification has a fundamental impact for the deduction of the rational approximants (...mas detalle...)
%
This modification has a fundamental impact for the deduction
of rational approximants to $\zeta \left( 3\right) $ that leads to a series
representation for Ap\'ery's constant, which converges faster than some series
proposed by several other authors. In order to complete this, we will deduce a
Hermite-Pad\'e approximation problem using Sorokin's ideas \cite%
{sorokin-1,van-assche-1}.
By this approach we obtain new rational approximants to $\zeta \left(
3\right) $ that prove its irrationality, but where Ap\'ery's phenomenon does
not appear. 

In Section \ref{SeccRR} we deduce a new recurrence relation as well as a new
continued fraction expansion connected to $\zeta \left( 3\right) $. Finally,
in Section \ref{SeccCSR} the convergence rate of several series
representations of $\zeta \left( 3\right) $ are compared.

\section{Ap\'ery's phenomenon} \label{sec:hist}

During several years Ap\'ery's phenomenon was interpreted by prestigious
mathematicians from the point of view of different analytic methods. In
1979, a few months after the appearance of Ap\'ery's celebrated proof of the
irrationality of $\zeta \left( 3\right) $, the Dutch mathematician Frits
Beukers (1979) interpreted Ap\'ery's phenomenon expressing the sequence of rational
approximations to $\zeta \left( 3\right) $ in terms of a triple integral \cite%
{beuker-1,beuker-2,Nesterenko2}%

\begin{eqnarray}
q_{n}\zeta \left( 3\right) -p_{n} &=&-\int_{0}^{1}\int_{0}^{1}\fracd{\log xy}{%
	1-xy}L_{n}\left( x\right) L_{n}\left( y\right) \,dxdy  \label{InBeukers} \\
&=&\int_{0}^{1}\int_{0}^{1}\int_{0}^{1}\fracd{\left( xyz\left( 1-x\right)
	\left( 1-y\right) \left( 1-z\right) \right) ^{n}}{\left( 1-\left(
	1-xy\right) z\right) ^{n+1}}\,dxdydz,  \label{BI2}
\end{eqnarray}%
where $n\in \mathbb{N}$ and % $L_{n}\left( x\right) $ 
\begin{equation*}
L_{n}\left( x \right) \equiv \fracd{1}{n!}\fracd{d^{n}}{dx^{n}} x^{n}\left(
1-x\right)^{n}=\sum_{0\leq k\leq n}\left( -1\right) ^{k}\binom{n+k}{k}%
\binom{n}{k} x^{k},
\end{equation*}%
it is the Legendre-type polynomial, orthogonal with respect to the Lebesgue measure on $\left(
0,1\right) $.

Moreover, Beukers (1979) showed that \eqref{InBeukers} behaves
as $\mathcal{O}\left( \varpi ^{-4n}\right) $, which proves Ap\'ery's theorem.
It is important to emphasize that this proof of irrationality of $\zeta
\left( 3\right) $ published by Beukers (1979) is much simpler and more
comprehensible compared to the original proof given by Ap\'ery. This approach
was continued in later works \cite{DViola,Hata1,Hata2,Hata3,RViola,zudilin-2}.%

Indeed, for $\zeta(2)$ Rhin and Viola (1996) consider for $\zeta(3)$ a generalization of%
\begin{equation*}
\int_{0}^{1}\int_{0}^{1}\fracd{\left( xy\left( 1-x\right) \left( 1-y\right)
	\right) ^{n}}{\left( 1-xy\right) ^{n+1}}\,dxdy\in \mathbb{Q-Z}\zeta \left(
2\right) ,
\end{equation*}%
given by%
\begin{equation*}
\int_{0}^{1}\int_{0}^{1}\fracd{x^{h}\left( 1-x\right) ^{i}y^{j}\left(
	1-y\right) ^{k}}{\left( 1-xy\right) ^{i+j-l+1}}\,dxdy\in \mathbb{Q-Z}\zeta
\left( 2\right) ,
\end{equation*}%
which depends on the five non negative parameters $h$, $i$, $j$, $k$ and $l$%
, see \cite{RViola1}\ for more details. Later, the same authors (2001)
consider a family of integrals generalizing $\left( \text{\ref{BI2}}\right) $
by%
\begin{equation*}
\int_{0}^{1}\int_{0}^{1}\int_{0}^{1}\fracd{x^{h}\left( 1-x\right)
	^{l}y^{k}\left( 1-y\right) ^{s}z^{j}\left( 1-z\right) ^{q}}{\left( 1-\left(
	1-xy\right) z\right) ^{q+h-r+1}}\,dxdydz\in \mathbb{Q-}2\mathbb{Z}\zeta
\left( 3\right) ,
\end{equation*}%
which depends on eight non negative parameters $h$, $j$, $k$, $l$, $m$, $q$, 
$r$ and $s$, subject to the conditions $j+q=l+s$, and $m=k+r-h$, see \cite%
{RViola} for more details. Indeed, these results combined with the group
method improve the irrationality measures \cite{Zudilin-7} for $\zeta \left(
2\right) $ and $\zeta \left( 3\right) $.

Two years after Ap\'ery's result, Beukers (1981) considered the following rational
approximation problem in an attempt to formulate Ap\'ery's proof in a more
natural way. It consisted in finding the polynomials $A_{n}\left( z\right)$
, $B_{n}\left( z\right) $, $C_{n}\left( z\right) $ and $D_{n}\left( z\right) 
$ of degree $n$ such that%
\begin{align*}
& A_{n}\left( z\right) \text{Li}_{1}\left( z\right) +B_{n}\left( z\right) 
\text{Li}_{2}\left( z\right) +C_{n}\left( z\right) =\mathcal{O}\left(
z^{2n+1}\right) , \\
& A_{n}\left( z\right) \text{Li}_{2}\left( z\right) +2B_{n}\left( z\right) 
\text{Li}_{3}\left( z\right) +D_{n}\left( z\right) =\mathcal{O}\left(
z^{2n+1}\right) ,
\end{align*}%
where%
\begin{equation*}
\text{Li}_{n}\left( z\right) :=\sum_{k\geq 1}\fracd{z^{k}}{k^{n}},
\end{equation*}%
denotes the polylogarithm of order $n$. Thereafter, Beukers introduced the
rational function%
\begin{equation}
\mathcal{B}_{n}\left( z\right) :=\fracd{\left( n-z+1\right) _{n}^{2}}{\left(
	-z\right) _{n+1}^{2}},  \label{RF-Beukers}
\end{equation}%
from which he deduced Ap\'ery's rational approximants $\left( \text{\ref{RA_Ap}%
}\right) $ by computing a partial fraction expansion, see \cite{beuker-3} for
more details. Here, $\left( \cdot \right) _{n}$ denotes the Pochhammer
symbol defined by%
\begin{align*}
(z)_{k}& :=\prod_{0\leq j\leq k-1}\left( z+j\right) ,\quad k\geq 1, \\
\left( z\right) _{0}& =1,\quad \left( -z\right) _{k}=0,\quad \mbox{if }z<k,
\end{align*}%
which is also called the shifted factorial and in terms of the gamma
function given by%
\begin{equation*}
(z)_{k}=\fracd{\Gamma \left( z+k\right) }{\Gamma \left( z\right) },\quad
k=0,1,2,\ldots
\end{equation*}%
On the other hand, in 1983, Ap\'ery's phenomenon was interpreted by Gutnik
(1983) in terms of Meijer's G-functions \cite{Luke}, i.e., 
\begin{equation*}
q_{n}\zeta \left( 3\right) -p_{n}=\,G_{4,4}^{4,2}\left( 
\begin{array}{c|c}
-n,-n,n+1,n+1 &  \\ 
& 1 \\ 
0,\text{ \ \ \ }0,\text{ \ \ \ }0,\text{ \ \ \ }0 & 
\end{array}%
\right) .
\end{equation*}%
This approach allowed him to prove several partial results on the
irrationality of certain quantities involving $\zeta \left( 2\right) $ and $%
\zeta \left( 3\right) $, see \cite{gutnik}\ for more details.

Later, Sorokin (1993) obtained Ap\'ery's rational approximants \eqref{RA_Ap} 
% $\left(\text{\ref{RA_Ap}}\right) $ 
in a similar way as Beukers (1979), by considering the approximation problem%
\begin{align*}
& A_{n}\left( z\right) f_{1}\left( z\right) +B_{n}\left( z\right)
f_{2}\left( z\right) -C_{n}\left( z\right) =\mathcal{O}\left(
z^{-n-1}\right) , \\
& A_{n}\left( z\right) f_{2}\left( z\right) +2B_{n}\left( z\right)
f_{3}\left( z\right) -D_{n}\left( z\right) =\mathcal{O}\left(
z^{-n-1}\right) ,
\end{align*}%
where $A_{n}\left( z\right) $ and $B_{n}\left( z\right) $ are polynomials of
degree $n$ and%
\begin{equation*}
f_{1}\left( z\right) =\int_{0}^{1}\fracd{dx}{z-x}\,dx,\quad f_{2}\left(
z\right) =-\int_{0}^{1}\fracd{\log x}{z-x}\,dx,\quad f_{3}\left( z\right) =%
\fracd{1}{2}\int_{0}^{1}\fracd{\log ^{2}x}{z-x}\,dx.
\end{equation*}%
Thus, he proved that the solution of this problem is given by the
orthogonality relations%
\begin{equation*}
\begin{array}{cc}
\displaystyle\int_{0}^{1}\left( A_{n}\left( x\right) -B_{n}\left( x\right)
\log x\right) x^{k}\,dx=0, &  \\ 
& \quad k=0,\ldots ,n-1, \\ 
\displaystyle\int_{0}^{1}\left( \left( A_{n}\left( x\right) -B_{n}\left(
x\right) \log x\right) \log x\right) x^{k}\,dx=0, & 
\end{array}%
\end{equation*}%
together with the additional condition $A_{n}\left( 1\right) =0$. Then,
using the Mellin convolution \cite{Smet,van-assche-1,van-assche-2} he obtains%
\begin{eqnarray*}
	q_{n}\zeta \left( 3\right) -p_{n} &=&\int_{0}^{1}\fracd{\left( A_{n}\left(
		x\right) -B_{n}\left( x\right) \log x\right) \log x}{1-x}\,dx \\
	&=&-\int_{0}^{1}\int_{0}^{1}\fracd{\log xy}{1-xy}L_{n}\left( x\right)
	L_{n}\left( y\right) \,dxdy,
\end{eqnarray*}%
which implies the irrationality of $\zeta \left( 3\right) $ according to
Beukers' estimation given by $\left( \text{\ref{InBeukers}}\right) $, see
for instance \cite{beuker-1,sorokin-1,van-assche-1}.

After this, inspired by Gutnik (1983), Nesterenko (1996) proposed a new
proof of Ap\'ery's theorem. For this purpose, he considered the modification
\begin{align} \label{nest96}
\mathcal{N}_{n}\left( z\right) :=\mathcal{B}_{n}\left( z+n+1\right) 
= \fracd{(-z)_n^2}{(z+1)_{n+1}^2},
\end{align}	
of
Beukers' rational function $\left( \text{\ref{RF-Beukers}}\right) $ and
proved the % following 
expression%
\begin{equation}
q_{n}\zeta \left( 3\right) -p_{n}=-\sum_{k\geq 0}\fracd{d}{dk}\mathcal{N}%
_{n}\left( k\right) =\fracd{1}{2\pi \mathfrak{i}}\int_{L}\mathcal{N}%
_{n}\left( \nu \right) \left( \fracd{\pi }{\sin \pi \nu }\right) ^{2}d\nu ,
\label{CIntegral}
\end{equation}%
for the error-term sequence, where $L$ is the vertical line Re $z=C$, $%
0<C<n+1$, oriented from top to bottom and%
\begin{equation*}
\fracd{d}{dz}\mathcal{N}_{n}\left( z\right) =2\mathcal{N}_{n}\left( z\right)
\left( \sum_{0\leq k\leq n-1}\fracd{1}{t-k}-\sum_{1\leq k\leq n+1}\fracd{1}{t+k%
}\right) .
\end{equation*}%
Indeed, the use of Laplace's method allowed him to estimate the above
contour integral $\left( \text{\ref{CIntegral}}\right) $ yielding the
behavior $\mathcal{O}\left( \varpi ^{-4n}\right) $, see \cite{Nesterenko2}
for more details. Moreover, he discovered a new continued fraction expansion
for $\zeta \left( 3\right) $ using the so-called Meijer functions \cite{Luke}%
, which have the form%
\begin{equation*}
2\zeta \left( 3\right) =2+\fracd{1\mid }{\mid 2}+\fracd{2\mid }{\mid 4}+\fracd{%
	1\mid }{\mid 3}+\fracd{4\mid }{\mid 2}+\fracd{2\mid }{\mid 4}+\fracd{6\mid }{%
	\mid 6}+\fracd{4\mid }{\mid 5}+\cdots ,
\end{equation*}%
where the numerators $a_{n}$, $n\geq 2$, and denominators $b_{n}$, $n\geq 1$%
, are definded by%
\begin{align*}
& b_{4k+1}=2k+2,\quad a_{4k+1}=k\left( k+1\right) ,\quad b_{4k+2}=2k+4, \\
& a_{4k+2}=\left( k+1\right) \left( k+2\right) ,\quad b_{4k+3}=2k+3, \\
& a_{4k+3}=\left( k+1\right) ^{2},\quad b_{4k}=2k,\quad a_{4k}=\left(
k+1\right) ^{2}.
\end{align*}%
In the same year, Pr\'evost (1996) published a new way of interpreting Ap\'ery's
phenomenon by recovering Ap\'ery's sequences using Pad\'e approximations to the
asymptotic expansion of the partial sum of $\zeta \left( 3\right) $ and
proving that 
\begin{equation*}
\left\vert q_{n}\zeta \left( 3\right) -p_{n}\right\vert \leq \fracd{4\pi ^{2}%
}{\left( 2n+1\right) ^{2}}\binom{n+k}{k}^{-2}\binom{n}{k}^{-2}.
\end{equation*}%
Based on the hypergeometric ideas of Nesterenko \cite{Nesterenko1}, Rivoal
and Ball \cite{BR,Rivoal}, and on Zeilberger's algorithm of creative
telescoping \cite{Zeilber}, Zudilin (2002) connected Ap\'ery's rational
aproximants with the following `very-well-posed hypergeometric series' \cite{Gasper,zudilin-2}%
\begin{multline*}
q_{n}\zeta \left( 3\right) -p_{n}=\fracd{n!^{7}\left( 3n+2\right) !}{\left(
	2n+1\right) !^{5}} \\
\times \,_{7}F_{6}\left( 
\begin{array}{c|c}
3n+2,\fracd{3n}{2}+2,n+1,\ldots ,n+1 &  \\ 
& 1 \\ 
\fracd{3n}{2}+1,2n+2,\ldots ,2n+2 & 
\end{array}%
\right) <20\left( n+1\right) ^{4}\varpi ^{-4n},
\end{multline*}%
which allowed him to prove the irrationality of $\zeta \left( 3\right) $,
see \cite{zudilin-1} for more details. Here, $_{r}F_{s}$ denotes the
ordinary hypergeometric series \cite{Gasper,Koekoek,Nikiforov} at the
variable $z$ defined by%
\begin{equation}
_{r}F_{s}\left( 
\begin{array}{c|c}
a_{1},\ldots ,a_{r} &  \\ 
& z \\ 
b_{1},\ldots ,b_{s} & 
\end{array}%
\right) :=\sum_{k\geq 0}\fracd{\left( a_{1}\right) _{k}\cdots \left(
	a_{r}\right) _{k}}{\left( b_{1}\right) _{k}\cdots \left( b_{s}\right) _{k}}%
\fracd{z^{k}}{k!}.  \label{rFs}
\end{equation}
Ap\'ery's phenomenon is not a necessary feature in alternative proofs of Ap\'ery's theorem.
There are also proofs of Ap\'ery's theorem, where Ap\'ery's phenomenon does not
appear.

Zulidin (2002) deduced a new sequence of rational approximants $\left\{ 
\tilde{p}_{n}/\tilde{q}_{n}\right\} $ to $\zeta \left( 3\right) $, whose
numerator $\tilde{p}_{n}$ and denominator $\tilde{q}_{n}$ satisfy the
recurrence relation%
\begin{multline}
\left( n+1\right) ^{4}\varphi _{0}\left( n\right) y_{n+1}-\varphi _{1}\left(
n\right) y_{n}+4\left( 2n-1\right) \varphi _{2}\left( n\right) y_{n-1} \\
-4\left( n-1\right) ^{2}\left( 2n-1\right) \left( 2n-3\right) \varphi
_{0}\left( n+1\right) y_{n-2}=0,  \label{RRZd}
\end{multline}%
with initial conditions%
\begin{equation*}
\tilde{p}_{0}=0,\quad \tilde{p}_{1}=17,\quad \tilde{p}_{2}=\fracd{9405}{8}%
,\quad \tilde{q}_{0}=1,\quad \tilde{q}_{1}=14,\quad \tilde{q}_{2}=978,\quad
\end{equation*}%
where%
\begin{equation*}
\varphi _{0}\left( n\right) =946n^{2}-731n+153,
\end{equation*}%
\begin{multline*}
\varphi _{1}\left( n\right) =2\bigl(%
104060n^{6}+127710n^{5}+12788n^{4}-34525n^{3}-8482n^{2}\\+3298n+1071\bigr),
\end{multline*}%
and%
\begin{equation*}
\varphi _{2}\left( n\right) =3784n^{5}-1032n^{4}-1925n^{3}+853n^{2}+328n-184.
\end{equation*}%
Here, the approach does not show Ap\'ery's phenomenon, since the
characteristic equation of $\left( \text{\ref{RRZd}}\right) $ does not
coincide with that one obtained by Ap\'ery and the rational approximants do
not prove the irrationality $\zeta \left( 3\right) $, see \cite{zudilin-5}
for more details.

In addition, Nesterenko (2009) published a new proof of the irrationality of 
$\zeta \left( 3\right) $. In this work, he proved that 
\begin{multline*}
\left( -1\right) ^{n}\mathcal{L}_{n}^{3}\sum_{k\geq 1}\fracd{\partial }{%
	\partial k}\left( k^{-2}\prod_{j=1}^{\left[ \left( n-1\right) /2\right] }%
\fracd{k-j}{k+j}\prod_{j=1}^{\left[ n/2\right] }\fracd{k-j}{k+j}\right) = \\
\left( -1\right) ^{n-1}\mathcal{L}_{n}^{3}\left( 2\mathcal{D}_{n}\zeta
\left( 3\right) -\mathcal{J}_{n}\right) <\left( 4/5\right) ^{n},
\end{multline*}%
where $\mathcal{D}_{n}$ and $\mathcal{J}_{n}$ are defined in \cite[eq. 5]%
{Nesterenko3}. From this statement, the irrationality of $\zeta \left(
3\right) $ can be proven. Here, Ap\'ery's phenomenon does not appear, since
neither the rational approximants nor the irrationality measure coincide
with Ap\'ery's results~\cite{Nesterenko3}.

Most recently \cite{Soria2}, from a modification of the rational function $\mathcal{N}%
_{n}\left( z\right) $, Soria-Lorente (2014) deduced the recurrence relation%
\begin{multline}
(n+2)^{4}\left( 24n^{3}+30n^{2}+16n+3\right) y_{n+2} \\
-4(n+1)(204n^{6}+1173n^{5}+2668n^{4}+3065n^{3} \\
+1905n^{2}+634n+86)y_{n+1} \\
+n^{4}\left( 24n^{3}+102n^{2}+148n+73\right) y_{n}=0,\quad n\geq 1,
\label{RR2}
\end{multline}%
which is satisfied by the numerators $\hat{p}_{n}$ and denominators $\hat{q}%
_{n}$ of the diophantine approximations to $\zeta \left( 3\right) $ given by%
\begin{equation}
\hat{q}_{n}=\ \sum_{0\leq k\leq n}d_{k}^{\left( n\right) }\quad \mbox{and}%
\quad \hat{p}_{n}=\sum_{1\leq k\leq n}d_{k}^{\left( n\right) }H_{k}^{\left(
	3\right) }+2^{-1}\sum_{1\leq k\leq n}c_{k}^{\left( n\right) }H_{k}^{\left(
	2\right) },  \label{RAA}
\end{equation}%
where%
\begin{eqnarray*}
	d_{k}^{\left( n\right) } &=&n^{-1}\binom{n+k-1}{k}^{2}\binom{n}{k}^{2} 
	+n^{-1}\binom{n+k-1}{k}^{2}\binom{n-1}{k-1}\binom{n}{k}, \\[1mm]
	c_{k}^{\left( n\right) } &=&2d_{k}^{\left( n\right) }\left[
	2H_{k}-H_{n+k-1}-H_{n-k}-2^{-1}\left( n+k\right) ^{-1}\right] ,\quad
	k=0,\ldots ,n,
\end{eqnarray*}%
and $H_{k}^{\left( r\right) }$ denotes the harmonic number $k$ of order $r$
defined by%
\begin{equation}
H_{k}^{\left( r\right) }=\sum_{1\leq j\leq k}\fracd{1}{j^{r}}.  \label{HA}
\end{equation}%
Hence, the irregular continued fraction expansion%
\begin{equation*}
\zeta \left( 3\right) =\fracd{7\mid }{\mid 6}+\fracd{\ -146\mid }{\mid \text{
		\ }827\text{ \ }}+\fracd{-38864\mid }{\mid \text{ \ \ }\mathcal{Q}_{3}\text{
		\ }}+\fracd{\mathcal{P}_{4}\mid }{\mid \mathcal{Q}_{4}}+\cdots +\fracd{%
	\mathcal{P}_{n}\mid }{\mid \mathcal{Q}_{n}}+\cdots ,
\end{equation*}%
could be derived, where%
\begin{multline*}
\mathcal{P}_{n}=-(n-2)^{4}(n-1)^{4}\left( 24n^{3}-186n^{2}+484n-423\right) \\
\times \left( 24n^{3}-42n^{2}+28n-7\right) ,
\end{multline*}%
and%
\begin{multline*}
\mathcal{Q}_{n}=4(n-1) \\
\times \left(
204n^{6}-1275n^{5}+3178n^{4}-3999n^{3}+2667n^{2}-910n+126\right) ,
\end{multline*}%
as well as the following series expansion%
\begin{equation*}
\zeta \left( 3\right) =\fracd{7}{6}+\sum_{n\geq 1}\fracd{24n^{3}+30n^{2}+16n+3%
}{2n^{3}\left( n+1\right) ^{3}{\Theta }_{n}{\Theta }_{n+1}},
\end{equation*}%
with%
\begin{equation*}
{\Theta }_{n}=\,_{4}F_{3}\left( 
\begin{array}{c|c}
-n,-n,n,n+1 &  \\ 
& 1 \\ 
1,1,1 & 
\end{array}%
\right) .
\end{equation*}%
Observe that the characteristic equation of $\left( \text{\ref{Apery-equat}}%
\right) $ coincides with that of $\left( \text{\ref{RR2}}\right) $, which is 
$\lambda ^{2}-34\lambda +1=0$, and its zeros are $\lambda _{1}=\varpi ^{4n}$
and $\lambda _{2}=\varpi ^{-4n}$ respectively. Hence, from Poincar\'e's
theorem \cite{perron,poincare} it has the behavior $\hat{q}_{n}=\mathcal{O}%
\left( \varpi ^{4n}\right) $ and $\hat{q}_{n}\zeta \left( 3\right) -\hat{p}%
_{n}=\mathcal{O}\left( \varpi ^{-4n}\right) $, as $n$ goes to infinity,
which proves Ap\'ery's theorem. Moreover, in such an instance, the
corresponding irrationality measure also coincides with the one obtained by Ap\'ery, see also \cite{Arvesu2}. However, Ap\'ery's phenomenon does not appear in
this case, since the rational approximants $\left( \text{\ref{RAA}}\right) $
to $\zeta \left( 3\right) $ do not coincide with $\left( \text{\ref{RA_Ap}}%
\right) $.

\section{Hermite-Pad\'e approximation problem connected to $\protect\zeta(3)$%
	\label{SeccHP}}

Our interest in this Section is to get an Hermite-Pad\'e approximation problem
connected to $\zeta \left( 3\right) $, from which in the following section
we deduce a new continued fraction expansion as well as a new series
representation to $\zeta \left( 3\right) $. For this purpose, inspired by
the results obtained by Sorokin (1998) (\cite{sorokin-1}, see also  \cite{RViola1,RViola,Soria2,van-assche-1}),
we introduce % will consider 
the following modification of the rational function $%
\mathcal{N}_{n}\left( z\right) $ defined by%
\begin{equation}
\mathcal{F}_{n,1}^{\left( \rho \right) }\left( z\right) :=\mathcal{N}%
_{n}\left( z\right) \left( \fracd{z-\rho n}{z-n+1}\right) =\fracd{\left(
	-z\right) _{n-1}^{2}\left( z-n+1\right) \left( z-\rho n\right) }{\left(
	z+1\right) _{n+1}^{2}},  \label{NNest}
\end{equation}%
which consists in changing the simple zero $z=n-1$ of the rational function $%
\mathcal{N}_{n}\left( z\right) $, 
\eqref{nest96} by the zero $z=\rho n$, with $\rho \in 
\mathbb{N}$. 
Because of its specific form, we refer to
$\mathcal{F}_{n,1}^{\left( \rho \right)}$
as the Nesterenko-like rational function.
For abbreviation we denote%
\begin{equation*}
\mathcal{F}_{n,2}^{\left( \rho \right) }\left( z\right) =\fracd{d}{dz}%
\mathcal{F}_{n,1}^{\left( \rho \right) }\left( z\right) .
\end{equation*}

\begin{lemma}
	\label{ILemm}Let $\rho \mathcal{\ }$be an integer number. Then, the
	following relation 
	\begin{equation}
	\displaystyle\mathcal{F}_{n,i}^{\left( \rho \right) }\left( z\right)
	=\int_{0}^{1}\psi _{n,i}^{\left( \rho \right) }\left( x\right)
	x^{z}\,dx,\quad i=1,2,  \label{R_Functions}
	\end{equation}%
	holds, where 
	\begin{equation}
	\psi _{n,1}^{\left( \rho \right) }\left( x \right) :=\ A_{n}^{\left( \rho
		\right) }\left( x \right) -B_{n}^{\left( \rho \right) }\left( x \right) \log
	x \quad \mbox{and}\quad \psi _{n,2}^{\left( \rho \right) }\left( x \right)
	:=\psi _{n,1}^{\left( \rho \right) }\left( x \right) \log x ,  \label{Psis}
	\end{equation}%
	being $A_{n}^{\left( \rho \right) }\left( x \right) $ and 
	$B_{n}^{\left( \rho \right) }\left( x \right) $\ polynomials of degree exactly $n$ defined by 
	\begin{equation}
	A_{n}^{\left( \rho \right) }\left( x \right) :=\ \sum_{0\leq k\leq
		n}a_{k,n}^{\left( \rho \right) } x^{k}\quad \mbox{and}\quad B_{n}^{\left(
		\rho \right) }\left( x \right) :=\ \sum_{0\leq k\leq n}b_{k,n}^{\left( \rho
		\right) } x^{k},  \label{Poly1}
	\end{equation}%
	with 
	%
	% \begin{multline}
	\begin{align}
	\begin{split}
	a_{k,n}^{\left( \rho \right) }=2b_{k,n}^{\left( \rho \right) } % \\
	% \times 
	\left[ 2H_{k}-H_{n+k-1}-H_{n-k}+\fracd{\left( \rho +1 \right) n+2k+1}{%
		2\left( k+\rho n+1\right) \left( n+k\right) }\right] , \\ % \quad k=0,\ldots ,n.
	% \label{aes}
	% \end{multline}%
	%
	% and
	%
	% \begin{equation}
	b_{k,n}^{\left( \rho \right) }=\binom{n+k}{k}^{2}\binom{n}{k}^{2}\left(
	k+\rho n+1\right) \left( n+k\right) ^{-1},\quad k=0,\ldots ,n, \label{bes}
	% \end{equation}%
	\end{split}
	\end{align}
	where $H_{k}^{\left( r\right) }$ denotes the harmonic number $k$ of order $r$
	given by $\eqref{HA}$.
\end{lemma}
Though the expression of the orthonogonality relation \eqref{R_Functions}
resembles that of Sorokin (1993) the 
the approximants of Sorokin coincide with Ap\'erys approximants,
whereas the new approximants \eqref{Poly1} do not.

\begin{proof}
	In fact, let us expand the functions $\mathcal{F}_{n,1}^{\left( \rho \right)
	}\left( z\right) $ and $\mathcal{F}_{n,2}^{\left( \rho \right) }\left(
	z\right) $ on the sum of partial fractions 
	\begin{equation}
	\mathcal{F}_{n,i}^{\left( \rho \right) }\left( z\right) =\ \left( -1\right)
	^{\delta _{i,2}}\sum_{0\leq k\leq n}\left( \fracd{\tilde{a}_{k,n}^{\left(
			\rho \right) }}{\left( z+k+1\right) ^{1+\delta _{i,2}}}+\fracd{2^{\delta
			_{i,2}}\tilde{b}_{k,n}^{\left( \rho \right) }}{\left( z+k+1\right) ^{\delta
			_{i,2}+2}}\right) ,  \label{SPF}
	\end{equation}%
	with $i=1,2$. 
	%
	%
	% El mtodo de descomposicin en fracciones parciales que utilizo aqu es el mtodo usual que se ensea el los libros de anlisis matemticos, fjate que tanto el numerador como el denominador de \mathcal{F}_{n,1} 
	% son polinomios reducibles. En caso del denominador, que es en realidad en lo hay que centrarse (fjate en la definicin del pochhammer al final de la pgina 5) el mismo se puede escribir en la forma [(x+1)(x+2)***(x+n+1)]^2, de ah la descomposicin en fracciones parciales para  \mathcal{F}_{n,1} y en caso de  \mathcal{F}_{n,2} slo hay que derivar la descomposicin en fracciones parciales de  \mathcal{F}_{n,1} para conseguir la de la descomposicin en fracciones parciales para  \mathcal{F}_{n,2}. 
	% Aqu a los aes y a los bes no se les puede quitar el tilde pues precisamente el primer paso es determinar cul es su forma (la cual en un principio no se ha determinado), la misma se determina evidentemente a partir de las expresiones que le siguen a la expresin (22). Claro est que la coincidencia no es obvia pero eso slo son calculillos que se pueden hacer sin problema alguno.
	%
	%
	Clearly 
	\begin{equation*}
	\tilde{b}_{k,n}^{\left( \rho \right) }=\left( z+k+1\right) ^{2}\mathcal{F}%
	_{n,1}^{\left( \rho \right) }\left( z\right) \Big|_{z=-k-1}\quad \mbox{and}%
	\quad \tilde{a}_{k,n}^{\left( \rho \right) }=\mathop{\mathrm{Res}}_{z=-k-1}%
	\mathcal{F}_{n,1}^{\left( \rho \right) }\left( z\right) ,
	\end{equation*}%
	which coincide with $\left( \text{\ref{bes}}\right) $.
	% and $\left( \text{\ref{aes}}\right) $ respectively. 
	Here, $\displaystyle\mathop{\mathrm{Res}}%
	_{z=z_{0}}f\left( z\right) $ denotes the residue of $f\left( z\right) $ at $%
	z=z_{0}$. In addition, applying the identity 
	\begin{equation}
	\fracd{\left( -1\right) ^{j}j!}{\left( i+1\right) ^{j+1}}=\int_{0}^{1}x^{i}%
	\log ^{j}x\text{ }dx,  \label{Int-log}
	\end{equation}%
	to $\left( \text{\ref{SPF}}\right) $ we have for $i=1,2$%
	\begin{equation*}
	\mathcal{F}_{n,i}^{\left( \rho \right) }\left( z\right) =\ \sum_{0\leq k\leq
		n}\int_{0}^{1}\left( a_{k,n}^{\left( \rho \right) }x^{z+k}\log ^{\delta
		_{i,2}}x-b_{k,n}^{\left( \rho \right) }x^{z+k}\log ^{1+\delta
		_{i,2}}x\right) \,dx.
	\end{equation*}%
	Hence%
	\begin{equation*}
	\mathcal{F}_{n,i}^{\left( \rho \right) }\left( z\right) =\ \int_{0}^{1}\
	\left( \sum_{0\leq k\leq n}a_{k,n}^{\left( \rho \right) }x^{k}-\log x\
	\sum_{0\leq k\leq n}b_{k,n}^{\left( \rho \right) }x^{k}\right) x^{z}\, \log ^{\delta_{i,2}}x\, dx,
	\end{equation*}%
	which completes the proof.
\end{proof}

For abbreviation we denote%
\begin{equation*}
\mathcal{R}_{n,i}^{\left( \rho \right) }\left( z\right) =\left(
-% 2^{-1}
\fracd{1}{2}
\right) ^{\delta _{i,2}}\sum_{j\geq 0}z^{-j-1}\mathcal{F}%
_{n,i}^{\left( \rho \right) }\left( j\right) .
\end{equation*}

\begin{lemma}
	\label{LRes}Let $\mathbb{P}_{n}$ be the $(n+1)$-dimensional subspace of the
	linear space $\mathbb{P}$ of polynomials with complex coefficients. Then,
	the following relations hold%
	\begin{eqnarray*}
		\mathcal{R}_{n,i}^{\left( \rho \right) }\left( z\right) &=&\left( -1\right)
		^{1+\delta _{i,2}}B_{n}^{\left( \rho \right) }\left( z\right) f_{\delta
			_{i,2}+2}\left( z\right) \\
		&&+\left( -2^{-1}\right) ^{\delta _{i,2}}A_{n}^{\left( \rho \right) }\left(
		z\right) f_{1+\delta _{i,2}}\left( z\right) -C_{n,i}^{\left( \rho \right)
		}\left( z\right) \\
		&=&\left( -2^{-1}\right) ^{\delta _{i,2}}\int_{0}^{1}\fracd{\psi
			_{n,i}^{\left( \rho \right) }\left( x\right) }{z-x}\,dx,\quad i=1,2,\quad
		n=0,1,\dots ,
	\end{eqnarray*}%
	where 
	\begin{equation}
	f_{j}\left( z\right) =\fracd{1}{(j-1)!}\int_{0}^{1}\fracd{\log ^{j-1}x}{z-x}%
	\,dx,\quad j\in \mathbb{N},  \label{fk(z)}
	\end{equation}%
	as well as 
	\begin{equation}
	C_{n,i}^{\left( \rho \right) }\left( z\right) =\left( -2^{-1}\right)
	^{\delta _{i,2}}\int_{0}^{1}\fracd{\psi _{n,i}^{\left( \rho \right) }\left(
		z\right) -\psi _{n,i}^{\left( \rho \right) }\left( x\right) }{z-x}\,dx,\quad
	C_{n,i}^{\left( \rho \right) }\left( z\right) \in \mathbb{P}_{n}.
	\label{Poly2}
	\end{equation}%
	By $\delta _{i,j}$ we denote the Kronecker delta function.
\end{lemma}

\begin{proof}
	In fact, from $\left( \text{\ref{SPF}}\right) $ we get 
	\begin{multline*}
	\mathcal{R}_{n,i}^{\left( \rho \right) }\left( z\right) =\sum_{j\geq
		0}z^{-j-1}\sum_{0\leq k\leq n}\fracd{b_{k,n}^{\left( \rho \right) }}{\left(
		j+k+1\right) ^{\delta _{i,2}+2}} \\
	+2^{-\delta _{i,2}}\sum_{j\geq 0}z^{-j-1}\sum_{0\leq k\leq n}\fracd{%
		a_{k,n}^{\left( \rho \right) }}{\left( j+k+1\right) ^{1+\delta _{i,2}}}.
	\end{multline*}%
	Next, interchanging the sums we have%
	\begin{eqnarray*}
		\mathcal{R}_{n,i}^{\left( \rho \right) }\left( z\right) &=&\sum_{0\leq k\leq
			n}b_{k,n}^{\left( \rho \right) }z^{k}\sum_{j\geq 0}\fracd{z^{-\left(
				j+k+1\right) }}{\left( j+k+1\right) ^{2+\delta _{i,2}}} \\
		&&+2^{-\delta _{i,2}}\sum_{0\leq k\leq n}a_{k,n}^{\left( \rho \right)
		}z^{k}\sum_{j\geq 0}\fracd{z^{-\left( j+k+1\right) }}{\left( j+k+1\right)
			^{1+\delta _{i,2}}} \\
		&=&\sum_{0\leq k\leq n}b_{k,n}^{\left( \rho \right) }z^{k}\sum_{l\geq k+1}%
		\fracd{z^{-l}}{l^{2+\delta _{i,2}}}+2^{-\delta _{i,2}}\sum_{0\leq k\leq
			n}a_{k,n}^{\left( \rho \right) }z^{k}\sum_{l\geq k+1}\fracd{z^{-l}}{%
			l^{1+\delta _{i,2}}}.
	\end{eqnarray*}%
	Splitting the sum over $l$ as%
	\begin{equation*}
	\sum_{l\geq k+1}f\left( l\right) =\sum_{l\geq 1}f\left( l\right)
	-\sum_{1\leq l\leq k}f\left( l\right) =\left( \sum_{l\geq 1}-\sum_{1\leq
		l\leq k}\right) f\left( l\right) ,
	\end{equation*}%
	we deduce%
	\begin{multline*}
	\mathcal{R}_{n,i}^{\left( \rho \right) }\left( z\right) =\sum_{0\leq k\leq
		n}b_{k,n}^{\left( \rho \right) }z^{k}\left( \sum_{l\geq 1}-\sum_{1\leq l\leq
		k}\right) \fracd{z^{-l}}{l^{2+\delta _{i,2}}} \\
	+2^{-\delta _{i,2}}\sum_{0\leq k\leq n}a_{k,n}^{\left( \rho \right)
	}z^{k}\left( \sum_{l\geq 1}-\sum_{1\leq l\leq k}\right) \fracd{z^{-l}}{%
		l^{1+\delta _{i,2}}}.
	\end{multline*}%
	Evidently%
	\begin{multline*}
	\mathcal{R}_{n}^{\left( \rho \right) }\left( z\right) =\sum_{0\leq k\leq
		n}b_{k,n}^{\left( \rho \right) }z^{k}\sum_{l\geq 1}\fracd{z^{-l}}{l^{2+\delta
			_{i,2}}}+2^{-\delta _{i,2}}\sum_{0\leq k\leq n}a_{k,n}^{\left( \rho \right)
	}z^{k}\sum_{l\geq 1}\fracd{z^{-l}}{l^{1+\delta _{i,2}}} \\
	-\sum_{1\leq k\leq n}b_{k,n}^{\left( \rho \right) }z^{k}\sum_{1\leq l\leq k}%
	\fracd{z^{-l}}{l^{2+\delta _{i,2}}}-2^{-\delta _{i,2}}\sum_{1\leq k\leq
		n}a_{k,n}^{\left( \rho \right) }z^{k}\sum_{1\leq l\leq k}\fracd{z^{-l}}{%
		l^{1+\delta _{i,2}}}.
	\end{multline*}%
	Clearly, from%
	\begin{equation*}
	\sum_{n\geq 1}\fracd{z^{-n}}{n^{j}}=\left( -1\right) ^{j-1}f_{j}\left(
	z\right) ,
	\end{equation*}%
	we have%
	\begin{multline}
	\mathcal{R}_{n,i}^{\left( \rho \right) }\left( z\right) =\left( -1\right)
	^{1+\delta _{i,2}}B_{n}^{\left( \rho \right) }\left( z\right) f_{2+\delta
		_{i,2}}\left( z\right) +\left( -2^{-1}\right) ^{\delta _{i,2}}A_{n}^{\left(
		\rho \right) }\left( z\right) f_{1+\delta _{i,2}}\left( z\right) \\
	-\sum_{1\leq k\leq n}b_{k,n}^{\left( \rho \right) }z^{k}\sum_{1\leq l\leq k}%
	\fracd{z^{-l}}{l^{2+\delta _{i,2}}}-2^{-\delta _{i,2}}\sum_{1\leq k\leq
		n}a_{k,n}^{\left( \rho \right) }z^{k}\sum_{1\leq l\leq k}\fracd{z^{-l}}{%
		l^{1+\delta _{i,2}}}.  \label{RP}
	\end{multline}%
	Then, using $\left( \text{\ref{Int-log}}\right) $ as well as%
	\begin{multline*}
	\int_{0}^{1}\fracd{A_{n}^{\left( \rho \right) }\left( z\right) -A_{n}^{\left(
			\rho \right) }\left( x\right) }{z-x}\log ^{\delta _{i,2}}x\,dx % \\
	=
	\sum_{\substack{0\leq k\leq n \\ 0\leq j\leq k-1}}
	%
	% \sum_{0\leq k\leq n} 
	a_{k,n}^{\left( \rho \right) }
	% \sum_{0\leq j\leq k-1}
	%
	z^{k-j-1}\int_{0}^{1}x^{j}\log ^{\delta _{i,2}}x\,dx,
	\end{multline*}%
	and%
	\begin{multline*}
	\int_{0}^{1}\fracd{B_{n}^{\left( \rho \right) }\left( z\right) -B_{n}^{\left(
			\rho \right) }\left( x\right) }{z-x}\log ^{1+\delta _{i,2}}x\,dx % \\
	=
	\sum_{\substack{0\leq k\leq n \\ 0\leq j\leq k-1}}
	% \sum_{0\leq k\leq n}
	b_{k,n}^{\left( \rho \right) }
	% \sum_{0\leq j\leq k-1}
	z^{k-j-1}\int_{0}^{1}x^{j}\log^{1+\delta _{i,2}} x\,dx,
	\end{multline*}%
	we deduce%
	\begin{multline}
	-\sum_{1\leq k\leq n}b_{k,n}^{\left( \rho \right) }z^{k}\sum_{1\leq l\leq k}%
	\fracd{z^{-l}}{l^{2+\delta _{i,2}}}-2^{-\delta _{i,2}}\sum_{1\leq k\leq
		n}a_{k,n}^{\left( \rho \right) }z^{k}\sum_{1\leq l\leq k}\fracd{z^{-l}}{%
		l^{1+\delta _{i,2}}} \\
	=\left( -2^{-1}\right) ^{\delta _{i,2}}\int_{0}^{1}\fracd{\psi _{n,i}^{\left(
			\rho \right) }\left( x\right) -\psi _{n,i}^{\left( \rho \right) }\left(
		z\right) }{z-x}\,dx.  \label{RIO}
	\end{multline}%
	Therefore, substituting the above in $\left( \text{\ref{RP}}\right) $ we
	arrived at the first equality. Next, let us prove the second equality.
	According to $\left( \text{\ref{Psis}}\right) $ and $\left( \text{\ref{fk(z)}%
	}\right) $ we have%
	\begin{multline*}
	-2^{-1}\int_{0}^{1}\fracd{\psi _{n,2}^{\left( \rho \right) }\left( x\right) }{%
		z-x}dx=2^{-1}\int_{0}^{1}\fracd{A_{n}^{\left( \rho \right) }\left( z\right)
		-A_{n}^{\left( \rho \right) }\left( x\right) }{z-x}\log x\,dx \\
	-2^{-1}\int_{0}^{1}\fracd{B_{n}^{\left( \rho \right) }\left( z\right)
		-B_{n}^{\left( \rho \right) }\left( x\right) }{z-x}\log ^{2}x\,dx \\
	-2^{-1}A_{n}^{\left( \rho \right) }\left( z\right) f_{2}\left( z\right)
	+B_{n}^{\left( \rho \right) }\left( z\right) f_{3}\left( z\right) \\
	=2^{-1}\int_{0}^{1}\fracd{\psi _{n,2}^{\left( \rho \right) }\left( z\right)
		-\psi _{n,2}^{\left( \rho \right) }\left( x\right) }{z-x}\,dx-2^{-1}A_{n}^{%
		\left( \rho \right) }\left( z\right) f_{2}\left( z\right) +B_{n}^{\left(
		\rho \right) }\left( z\right) f_{3}\left( z\right) ,
	\end{multline*}%
	and%
	\begin{multline*}
	\int_{0}^{1}\fracd{\psi _{n,1}^{\left( \rho \right) }\left( x\right) }{z-x}%
	dx=-\int_{0}^{1}\fracd{A_{n}^{\left( \rho \right) }\left( z\right)
		-A_{n}^{\left( \rho \right) }\left( x\right) }{z-x}\,dx \\
	+\int_{0}^{1}\fracd{B_{n}^{\left( \rho \right) }\left( z\right)
		-B_{n}^{\left( \rho \right) }\left( x\right) }{z-x}\log x\,dx \\
	+A_{n}^{\left( \rho \right) }\left( z\right) f_{1}\left( z\right)
	-B_{n}^{\left( \rho \right) }\left( z\right) f_{2}\left( z\right) \\
	=\int_{0}^{1}\fracd{\psi _{n,1}^{\left( \rho \right) }\left( x\right) -\psi
		_{n,1}^{\left( \rho \right) }\left( z\right) }{z-x}\,dx+A_{n}^{\left( \rho
		\right) }\left( z\right) f_{1}\left( z\right) -B_{n}^{\left( \rho \right)
	}\left( z\right) f_{2}\left( z\right) .
	\end{multline*}%
	Thus, taking $\left( \text{\ref{RP}}\right) $ and $\left( \text{\ref{RIO}}%
	\right) $ into account, we obtain the desired result.
\end{proof}

Notice that, using the identity%
\begin{equation*}
\fracd{1}{z-x}=\sum_{0\leq j\leq n-1-\delta _{i,2}}\fracd{x^{j}}{z^{j+1}}+%
\fracd{x^{n-\delta _{i,2}}}{z^{n-\delta _{i,2}}}\fracd{1}{z-x},\quad i=1,2,
\end{equation*}%
as well as the previous lemma we have for $i=1,2$ the following%
\begin{multline*}
\left( -2\right) ^{\delta _{i,2}}\mathcal{R}_{n,i}^{\left( \rho \right)
}\left( z\right) \\
=\sum_{0\leq j\leq n-1-\delta _{i,2}}\fracd{1}{z^{k+1}}\int_{0}^{1}\psi
_{n,i}^{\left( \rho \right) }\left( x\right) x^{j}\,dx+\fracd{1}{z^{n-\delta
		_{i,2}}}\int_{0}^{1}\fracd{x^{n-\delta _{i,2}}\psi _{n,i}^{\left( \rho
		\right) }\left( x\right) }{z-x}\,dx \\
=\sum_{0\leq k\leq n-1-\delta _{i,2}}\fracd{1}{z^{j+1}}\int_{0}^{1}\psi
_{n,i}^{\left( \rho \right) }\left( x\right) x^{j}\,dx+\mathcal{O}\left(
z^{-n-\delta _{i,1}}\right) .
\end{multline*}%
Next, taking Lemma \ref{ILemm} into account, as well as the zeros of the
rational function $\left( \text{\ref{NNest}}\right) $, we infer the
following orthogonal conditions for $i=1,2$,%
\begin{equation}
\displaystyle\int_{0}^{1}\psi _{n,i}^{\left( \rho \right) }\left( x\right)
x^{j}\,dx=0,\quad j=0,\ldots ,n-\delta _{i,2}-1,  \label{Orth_Cs}
\end{equation}%
from which we see that $\mathcal{R}_{n,i}^{\left( \rho \right) }\left(
z\right) =\mathcal{O}\left( z^{-n-\delta _{i,1}}\right) $ for $i=1,2$.
Moreover, since $\mathcal{F}_{n,1}^{\left( \rho \right) }\left( z\right) =%
\mathcal{O}\left( z^{-2}\right) $ as $z\rightarrow \infty $, we deduce%
\begin{equation}
A_{n}^{\left( \rho \right) }\left( 1\right) =\sum_{0\leq k\leq n}%
\mathop{\mathrm{Res}}_{z=-k-1}\mathcal{F}_{n,1}^{\left( \rho \right) }\left(
z\right) =-\mathop{\mathrm{Res}}_{z=\infty }\mathcal{F}_{n,1}^{\left( \rho
	\right) }\left( z\right) =0.  \label{Azero}
\end{equation}%
Having in mind all the above results, we observe that the functions $\left( 
\text{\ref{R_Functions}}\right) $ and the polynomials $\left( \text{\ref%
	{Poly1}}\right) $ and $\left( \text{\ref{Poly2}}\right) $ are connected to
the following Hermite-Pad\'e approximation problem%
\begin{eqnarray*}
	\tilde{B}_{n}^{\left( \rho ,i\right) }\left( z\right) f_{\delta
		_{i,2}+2}\left( z\right) +\tilde{A}_{n}^{\left( \rho ,i\right) }\left(
	z\right) f_{1+\delta _{i,2}}\left( z\right) -C_{n,i}^{\left( \rho \right)
	}\left( z\right) &=&\mathcal{O}\left( z^{-n-\delta _{i,1}}\right) , \\
	A_{n}^{\left( \rho \right) }\left( 1\right) &=&0,
\end{eqnarray*}%
where $i=1,2$, $\tilde{B}_{n}^{\left( \rho ,i\right) }\left( z\right)
=\left( -1\right) ^{1+\delta _{i,2}}B_{n}^{\left( \rho \right) }\left(
z\right) $ and $\tilde{A}_{n}^{\left( \rho ,i\right) }\left( z\right)
=\left( -2^{-1}\right) ^{\delta _{i,2}}A_{n}^{\left( \rho \right) }\left(
z\right) $.

From the Hermite-Pad\'e approximation problem of the Lemma %
\ref{LRes} we can deduce Corollary \ref{cor}.

\begin{corollary} \label{cor}
	Let $n\geq 1$, then the following relation 
	\begin{equation}
	\mathcal{R}_{n,2}^{\left( \rho \right) }\left( 1\right) =B_{n}^{\left( \rho
		\right) }\left( 1\right) \zeta \left( 3\right) -C_{n,2}^{\left( \rho \right)
	}\left( 1\right) =-2^{-1}\int_{0}^{1}\fracd{\psi _{n,2}^{\left( \rho \right)
		}\left( x\right) }{1-x}dx,  \label{ResZIII1}
	\end{equation}%
	holds, where 
	\begin{equation*}
	C_{n,2}^{\left( \rho \right) }\left( 1\right) =\sum_{1\leq k\leq n}\left(
	b_{k,n}^{\left( \rho \right) }H_{k}^{3}+2^{-1}a_{k,n}^{\left( \rho \right)
	}H_{k}^{2}\right) .
	\end{equation*}
\end{corollary}
This corollary is a specific case of the Hermite-Pad\'e problem where $z=1$,
where
$B_{n}^{\left( \rho \right) }\left( 1\right)$ and $-C_{n,2}^{\left( \rho \right)}\left( 1 \right)$ 
are the denominators and numerators of the rational approximants of $\zeta(3)$, respectively,
and the $\mathcal{R}_{n,2}^{\left( \rho \right) }\left( 1\right)$ are the residuals.

% Esto es precisamente el caso particular del problema de Hermite-Pad cuando z=1, como te podrs percatar los  B_{n}^{\left( \rho \right) }\left( 1\right) y los -C_{n,2}^{\left( \rho \right) }\left( 1\right) son los denominadores y los numeradores de los aproximantes racionales a \zeta (3) respectivamente, y los \mathcal{R}_{n,2}^{\left( \rho \right) }\left( 1\right) son los residuos.

\section{Main results\label{SeccRR}}

In this Section the main results of this contribution are stated. We present a new recurrence relation as well
as a new continued fraction expansion and a new series expansions for $\zeta \left( 3\right) $, which depends on one single integer parameter.

With the abbreviations
\begin{equation}
\begin{array}{c}
\left( r_{n}^{\left( \rho \right) }\right) _{n\geq 1}=\left\{ \mathcal{R}%
_{n,2}^{\left( \rho \right) }\left( 1\right) \right\} _{n\geq 1},\quad
\left( q_{n}^{\left( \rho \right) }\right) _{n\geq 1}=\left\{ B_{n}^{\left(
	\rho \right) }\left( 1\right) \right\} _{n\geq 1}, \\ 
\\ 
\mbox{and}\quad \left( p_{n}^{\left( \rho \right) }\right) _{n\geq
	1}=\left\{ C_{n,2}^{\left( \rho \right) }\left( 1\right) \right\} _{n\geq 1},
\end{array}
\label{AppN}
\end{equation}%
equation $\left( \text{\ref{ResZIII1}}\right) $ can be rewritten as%
\begin{equation}
r_{n}^{\left( \rho \right) }=q_{n}^{\left( \rho \right) }\zeta \left(
3\right) -p_{n}^{\left( \rho \right) }.  \label{ResZIII2}
\end{equation}%
According to $\left( \text{\ref{ResZIII2}}\right) $ we deduce that%
\begin{equation}
p_{n}^{\left( \rho \right) }q_{n+1}^{\left( \rho \right) }-p_{n+1}^{\left(
	\rho \right) }q_{n}^{\left( \rho \right) }=q_{n}^{\left( \rho \right)
}r_{n+1}^{\left( \rho \right) }-q_{n+1}^{\left( \rho \right) }r_{n}^{\left(
	\rho \right) }.  \label{Rpnqnrn}
\end{equation}%
Notice that, as a consequence of Lemma \ref{LRes} and the orthogonality
conditions $\left( \text{\ref{Orth_Cs}}\right)$ we have%
\begin{equation}
\begin{array}{c}
\displaystyle\int_{0}^{1}\fracd{P_{n-1}\left( x\right) \psi _{n,1}^{\left(
		\rho \right) }\left( x\right) }{1-x}dx=P_{n-1}\left( 1\right) \int_{0}^{1}%
\fracd{\psi _{n,1}^{\left( \rho \right) }\left( x\right) }{1-x}dx, \\ 
\\ 
\displaystyle P_{n-1}\left( 1\right) r_{n}^{\left( \rho \right)
}=-2^{-1}\int_{0}^{1}\fracd{P_{n-1}\left( x\right) \psi _{n,2}^{\left( \rho
		\right) }\left( x\right) }{1-x}dx,%
\end{array}
\label{Cons}
\end{equation}%
where $P_{n-1}\left( x\right) $ is an arbitrary polynomial of degree at most 
$n-1$.

\begin{lemma}
	\label{LMRE}Let $\mathcal{F}_{n,1}^{\left( \rho \right) }\left( z\right) $
	be the rational function defined by $\left( \text{\ref{NNest}}\right) $.
	Then, the following relations hold%
	\begin{equation*}
	\mathcal{F}_{n,2}^{\left( \rho \right) }\left( n-1\right) =\fracd{\left( \rho
		n-n+1\right) \left( n-1\right) !^{4}}{\left( 2n\right) !^{2}},
	\end{equation*}%
	\begin{equation*}
	\mathcal{F}_{n,2}^{\left( \rho \right) }\left( n\right) =-\fracd{2n\left(
		\rho -1\right) n!^{2}}{\left( n+1\right) _{n+1}^{2}}\left( 2H_{n}-H_{2n+1}-%
	\fracd{\rho n-n+1}{2n\left( \rho -1\right) }\right) ,
	\end{equation*}%
	and%
	\begin{equation*}
	\mathcal{F}_{n,1}^{\left( \rho \right) }\left( n\right) =-\fracd{n\left( \rho
		-1\right) n!^{2}}{\left( n+1\right) _{n+1}^{2}},
	\end{equation*}
\end{lemma}

\begin{proof}
	To prove the Lemma it is enough to evaluate $\mathcal{F}_{n,2}^{\left( \rho
		\right) }\left( n\right) $ at $n-1$ and $n$. Indeed, the desired result
	follows from a tedious but straightforward verification.
\end{proof}

\begin{lemma}
	\label{DCasar}The sequences $\left( p_{n}^{\left( \rho \right) }\right)
	_{n\geq 1}$ and $\left( q_{n}^{\left( \rho \right) }\right) _{n\geq 1}$
	defined by $\left( \text{\ref{AppN}}\right) $ satisfy the following relation 
	\begin{equation}
	\det 
	\begin{pmatrix}
	p_{n}^{\left( \rho \right) } & q_{n}^{\left( \rho \right) } \\ 
	p_{n+1}^{\left( \rho \right) } & q_{n+1}^{\left( \rho \right) }%
	\end{pmatrix}%
	=-\fracd{\Phi _{n}^{\left( \rho \right) }}{2n^{4}\left( n+1\right) ^{4}}%
	,\quad \rho \in \mathbb{N},\quad n\geq 1,  \label{eqq_1}
	\end{equation}%
	where 
	\begin{multline}
	\Phi _{n}^{\left( \rho \right) }=24n^{5}\rho ^{2}-12n^{5}+54n^{4}\rho
	^{2}+39n^{4}\rho -33n^{4}+46n^{3}\rho ^{2}+70n^{3}\rho \\
	+19n^{2}\rho ^{2}+56n^{2}\rho +33n^{2}+3n\rho ^{2}+21n\rho +24n+3\rho +5.
	\label{Phi}
	\end{multline}
\end{lemma}

\begin{proof}
	In fact, using $\left( \text{\ref{Azero}}\right) $ as well as $\left( \text{%
		\ref{Cons}}\right) $ we get 
	\begin{eqnarray*}
		q_{n}^{\left( \rho \right) }r_{n+1}^{\left( \rho \right) }
		&=&2^{-1}A_{n}^{\left( \rho \right) }\left( 1\right) \int_{0}^{1}\fracd{\psi
			_{n+1,1}^{\left( \rho \right) }\left( x\right) }{1-x}dx-2^{-1}\int_{0}^{1}%
		\fracd{B_{n}^{\left( \rho \right) }\left( x\right) \psi _{n+1,2}^{\left( \rho
				\right) }\left( x\right) }{1-x}dx \\
		&=&2^{-1}\int_{0}^{1}\fracd{\psi _{n,1}^{\left( \rho \right) }\left( x\right)
			\psi _{n+1,1}^{\left( \rho \right) }\left( x\right) }{1-x}dx.
	\end{eqnarray*}%
	In addition 
	\begin{multline*}
	2^{-1}\int_{0}^{1}\fracd{\psi _{n,1}^{\left( \rho \right) }\left( x\right)
		\psi _{n+1,1}^{\left( \rho \right) }\left( x\right) }{1-x}%
	dx=2^{-1}\int_{0}^{1}\fracd{A_{n+1}^{\left( \rho \right) }\left( x\right)
		\psi _{n,1}^{\left( \rho \right) }\left( x\right) }{1-x}dx \\
	-2^{-1}\int_{0}^{1}\fracd{B_{n+1}^{\left( \rho \right) }\left( x\right) \psi
		_{n,2}^{\left( \rho \right) }\left( x\right) }{1-x}dx,
	\end{multline*}%
	where 
	\begin{multline*}
	\int_{0}^{1}\fracd{A_{n+1}^{\left( \rho \right) }\left( x\right) \psi
		_{n,1}^{\left( \rho \right) }\left( x\right) }{1-x}dx=-\sum_{0\leq k\leq
		n+1}a_{k,n+1}^{\left( \rho \right) }\sum_{1\leq j\leq
		k}\int_{0}^{1}x^{j-1}\psi _{n,1}^{\left( \rho \right) }\left( x\right) dx \\
	=-a_{n+1,n+1}^{\left( \rho \right) }\mathcal{F}_{n,1}^{\left( \rho \right)
	}\left( n\right) ,
	\end{multline*}%
	and 
	\begin{multline*}
	-2^{-1}\int_{0}^{1}\fracd{B_{n+1}^{\left( \rho \right) }\left( x\right) \psi
		_{n,2}^{\left( \rho \right) }\left( x\right) }{1-x}dx= \\
	2^{-1}\sum_{0\leq k\leq n+1}b_{k,n+1}^{\left( \rho \right) }\sum_{1\leq
		j\leq k}\int_{0}^{1}x^{j-1}\psi _{n,2}^{\left( \rho \right) }\left( x\right)
	dx+q_{n+1}^{\left( \rho \right) }r_{n}^{\left( \rho \right) }.
	\end{multline*}%
	Thus, taking the relations $\left( \text{\ref{R_Functions}}\right) $ and $%
	\left( \text{\ref{Rpnqnrn}}\right) $ into account, as well as the
	orthogonality conditions $\left( \text{\ref{Orth_Cs}}\right) $, we deduce 
	\begin{multline}
	p_{n}^{\left( \rho \right) }q_{n+1}^{\left( \rho \right) }-p_{n+1}^{\left(
		\rho \right) }q_{n}^{\left( \rho \right) }=2^{-1}b_{n,n+1}^{\left( \rho
		\right) }\mathcal{F}_{n,2}^{\left( \rho \right) }\left( n-1\right)
	+2^{-1}b_{n+1,n+1}^{\left( \rho \right) }\mathcal{F}_{n,2}^{\left( \rho
		\right) }\left( n-1\right) \\
	+2^{-1}b_{n+1,n+1}^{\left( \rho \right) }\mathcal{F}_{n,2}^{\left( \rho
		\right) }\left( n\right) -2^{-1}a_{n+1,n+1}^{\left( \rho \right) }\mathcal{F}%
	_{n,1}^{\left( \rho \right) }\left( n\right) .  \label{RFinal}
	\end{multline}%
	By considering Lemma \ref{LMRE} we conclude that $\left( \text{%
		\ref{RFinal}}\right)$ coincides with $\left( \text{\ref{eqq_1}}\right) $,
	which is the desired conclusion.
\end{proof}

Next, we apply the so-called algorithm of creative telescoping due to 
% W. Gosper and D. Zeilberger
Gosper and Zeilberger~\cite{Abramov1,Abramov2,Abramov3,Abramov4,Zeilber},
from which we deduce the first part of the proof. For cross-validation we implemented this algorithm in different computer algebra systems, in particular, in Maple, in Payton and Mathematica.

\begin{theorem}
	\label{Proposition_AL_RR}Let $\left( p_{n}^{\left( \rho \right) }\right)
	_{n\geq 1}$, $\left( q_{n}^{\left( \rho \right) }\right) _{n\geq 1}$ \ and $%
	\left( r_{n}^{\left( \rho \right) }\right) _{n\geq 1}$ be the sequences
	defined by $\left( \text{\ref{AppN}}\right) $, where $\left( p_{n}^{\left(
		\rho \right) }\right) _{n\geq 1}$ and $\left( q_{n}^{\left( \rho \right)
	}\right) _{n\geq 1}$ satisfy the relation $\left( \text{\ref{eqq_1}}\right) $%
	. Then the following recurrence relation 
	\begin{equation}
	(n+2)^{4}\Phi _{n}^{\left( \rho \right) }y_{n+2}+\beta _{n}^{\left( \rho
		\right) }y_{n+1}+n^{4}\Phi _{n+1}^{\left( \rho \right) }y_{n}=0,\quad n\geq
	1,\quad \rho \in \mathbb{N},  \label{A_Like_RR}
	\end{equation}%
	holds, where 
	\begin{multline}
	\beta _{n}^{\left( \rho \right) }=-2(n+1)(408n^{8}\rho
	^{2}-204n^{8}+3162n^{7}\rho ^{2}+663n^{7}\rho \\
	-1683n^{7}+10028n^{6}\rho ^{2}+4433n^{6}\rho -4899n^{6}+16802n^{5}\rho ^{2}
	\\
	+12409n^{5}\rho -5487n^{5}+16070n^{4}\rho ^{2}+18955n^{4}\rho \\
	+735n^{4}+8888n^{3}\rho ^{2}+17212n^{3}\rho +7366n^{3} \\
	+2708n^{2}\rho ^{2}+9340n^{2}\rho +6870n^{2}+344n\rho ^{2} \\
	+2776n\rho +2748n+344\rho +412),  \label{Betan}
	\end{multline}%
	and $\Phi _{n}^{\left( \rho \right) }$ is given in $\left( \text{\ref{Phi}}%
	\right) $.
\end{theorem}

This recurrence relation has the special property that it depends only on $\rho$ as parameter.

% Lo que te puedo decir en este caso es que slo en este paper aparece dicha relacin de recurrencia con dicho \beta

\begin{proof}
	The proof will be divided into three steps. In fact, firstly let us prove
	that the sequence $\left( q_{n}^{\left( \rho \right) }\right) _{n\geq 1}$ satisfies the
	recurrence relation $\left( \text{\ref{A_Like_RR}}\right) $. For such
	propose, let us suppose that there exists other constants $\alpha
	_{n}^{\left( \rho \right) }$, $\hat{\beta}_{n}^{\left( \rho \right) }$ and $%
	\gamma _{n}^{\left( \rho \right) }$, which are not all equal to zero, such that 
	\begin{equation*}
	\alpha _{n}^{\left( \rho \right) }q_{n+2}^{\left( \rho \right) }+\hat{\beta}
	_{n}^{\left( \rho \right) }q_{n+1}^{\left( \rho \right) }+\gamma
	_{n}^{\left( \rho \right) }q_{n}^{\left( \rho \right) }=0,\quad n\geq 0 . 
	\end{equation*}
	This is equivalent to 
	\begin{equation*}
	\sum_{0\leq k\leq n+2}\left( \alpha _{n}^{\left( \rho \right)
	}b_{k,n+2}^{\left( \rho \right) }+\hat{\beta}_{n}^{\left( \rho \right)
	}b_{k,n+1}^{\left( \rho \right) }+\gamma _{n}^{\left( \rho \right)
	}b_{k,n}^{\left( \rho \right) }\right) =0,
	\end{equation*}%
	since $b_{j,k}^{\left( \rho \right) }=0$, for $j>k$.
	(Compare with \eqref{bes} for the definition of the 
	$b_{k,n}^{\left( \rho \right) }$ terms.)
	Therefore%
	\begin{equation}
	\alpha _{n}^{\left( \rho \right) }b_{k,n+2}^{\left( \rho \right) }+\hat{
		\beta }_{n}^{\left( \rho \right) }b_{k,n+1}^{\left( \rho \right) }+\gamma
	_{n}^{\left( \rho \right) }b_{k,n}^{\left( \rho \right) }=f_{n}\left(
	k+1\right) -f_{n}\left( k\right) ,  \label{Zeilb1,1}
	\end{equation}%
	such that $f_{n}\left( 0\right) =f_{n}\left( n+3\right) =0$. 
	According to the method of Zeilberger we can
	define%
	\begin{equation}
	f_{n}\left( k\right) =\fracd{k^{4}\pi _{3,n}\left( k\right) \binom{n+k}{k}
		^{2} \binom{n}{k}^{2}}{\left( n-k+1\right) ^{2}\left( n-k+2\right)
		^{2}\left( n+k\right) },  \label{fn1,1}
	\end{equation}%
	where $\pi _{3,n}\left( k\right)$ is a polynomial of degree $3$ in $k$, with
	coefficients depending on $n$. From $\left( \text{\ref{Zeilb1,1}}\right) $
	and $\left( \text{\ref{fn1,1}}\right) $ the following equation%
	\begin{multline*}
	\alpha _{n}^{\left( \rho \right) }\left( n+k\right) \left( n+k+1\right)
	^{2}\left( n+k+2\right) \left( k+\rho n+2\rho +1\right) \\
	+\hat{\beta}_{n}^{\left( \rho \right) }\left( n-k+2\right) ^{2}\left(
	n+k\right) \left( n+k+1\right) \left( k+\rho n+\rho +1\right) \\
	+\gamma _{n}^{\left( \rho \right) }\left( n-k+1\right) ^{2}\left(
	n-k+2\right) ^{2}\left( k+\rho n+1\right) \\
	=\left( n-k+2\right) ^{2}\left( n+k\right) \left( n+k+1\right) \pi
	_{3,n}\left( k+1\right) -k^{4}\pi _{3,n}\left( k\right) ,
	\end{multline*}%
	holds. The above leads to a $6$-equation linear system with $7$-unknowns. 
	A particular solution to this system 
	can be obtained by computer algebra 
	and is given by $\alpha _{n}^{\left( \rho
		\right) }=(n+2)^{4}\Phi _{n}^{\left( \rho \right) }$, $\gamma _{n}^{\left(
		\rho \right) }=n^{4}\Phi _{n+1}^{\left( \rho \right) }$ and $\hat{\beta}%
	_{n}^{\left( \rho \right) }=\beta _{n}^{\left( \rho \right) }$, which proves
	that the sequence $\left( q_{n}^{\left( \rho \right) }\right) _{n\geq 1}$ satisfies the
	recurrence relation $\left( \text{\ref{A_Like_RR}}\right)$.
	
	Our next goal is to prove that 
	the sequence $\left( r_{n}^{\left( \rho \right) }\right)_{n\geq 1}$ 
	satisfies the recurrence relation $\left( \text{\ref{A_Like_RR}}%
	\right) $. For this purpose let us use the Lemma \ref{DCasar}, from
	which we have%
	\begin{align*}
	& q_{n}^{\left( \rho \right) }r_{n+1}^{\left( \rho \right) }=q_{n+1}^{\left(
		\rho \right) }r_{n}^{\left( \rho \right) }-\fracd{\Phi _{n}^{\left( \rho
			\right) }}{2n^{4}\left( n+1\right) ^{4}}, \\
	& q_{n+1}^{\left( \rho \right) }r_{n+2}^{\left( \rho \right)
	}=q_{n+2}^{\left( \rho \right) }r_{n+1}^{\left( \rho \right) }-\fracd{\Phi
		_{n+1}^{\left( \rho \right) }}{2\left( n+1\right) ^{4}\left( n+2\right) ^{4}}%
	,
	\end{align*}%
	which is equivalent to%
	\begin{align*}
	& \fracd{q_{n}^{\left( \rho \right) }}{q_{n+1}^{\left( \rho \right) }}
	r_{n+1}^{\left( \rho \right) }=r_{n}^{\left( \rho \right) }-\fracd{\Phi
		_{n}^{\left( \rho \right) }}{2n^{4}\left( n+1\right) ^{4}q_{n+1}^{\left(
			\rho \right) }}, \\
	& r_{n+2}^{\left( \rho \right) }=\fracd{q_{n+2}^{\left( \rho \right) }}{
		q_{n+1}^{\left( \rho \right) }}r_{n+1}^{\left( \rho \right) }-\fracd{\Phi
		_{n+1}^{\left( \rho \right) }}{2\left( n+1\right) ^{4}\left( n+2\right)
		^{4}q_{n+1}^{\left( \rho \right) }}.
	\end{align*}%
	Thus, multiplying the first equation by $-n^{4}\Phi _{n+1}^{\left( \rho
		\right) }$, the second one by~$(n+2)^{4}\Phi _{n}^{\left( \rho \right)}$,
	and adding both equations we deduce%
	\begin{equation*}
	(n+2)^{4}\Phi _{n}^{\left( \rho \right) }r_{n+2}^{\left( \rho \right) }+ 
	\tilde{\beta}_{n}^{\left( \rho \right) }r_{n+1}^{\left( \rho \right)
	}+n^{4}\Phi _{n+1}^{\left( \rho \right) }r_{n}^{\left( \rho \right) }=0,
	\end{equation*}%
	where%
	\begin{equation*}
	\tilde{\beta}_{n}^{\left( \rho \right) }=-\fracd{n^{4}\Phi _{n+1}^{\left(
			\rho \right) }q_{n}^{\left( \rho \right) }}{q_{n+1}^{\left( \rho \right) }}- 
	\fracd{(n+2)^{4}\Phi _{n}^{\left( \rho \right) }q_{n+2}^{\left( \rho \right)
	} }{q_{n+1}^{\left( \rho \right) }},
	\end{equation*}
	which coincides with $\left( \text{\ref{Betan}}\right) $ since the sequence $%
	\left( q_{n}^{\left( \rho \right) }\right) _{n\geq 1}$ satisfies the
	recurrence relation $\left( \text{\ref{A_Like_RR}}\right) $. Therefore, we
	conclude that $\left( r_{n}^{\left( \rho \right) }\right) _{n\geq 1}$ also
	satisfies $\left( \text{\ref{A_Like_RR}}\right) $. Finally, the sequence $%
	\left( p_{n}^{\left( \rho \right) }=q_{n}^{\left( \rho \right) }\zeta \left(
	3\right) -r_{n}^{\left( \rho \right) }\right) _{n\geq 0}$ satisfies the
	recurrence relation $\left( \text{\ref{A_Like_RR}}\right) $ as a linear
	combination of the sequences $\left( q_{n}^{\left( \rho \right) }\right)
	_{n\geq 0}$ and $\left( r_{n}^{\left( \rho \right) }\right) _{n\geq 0}$.
	This completes the proof.
\end{proof}

Using the expressions%
\begin{equation*}
q_{n}^{\left( \rho \right) }=\sum_{1\leq k\leq n}b_{k,n}^{\left( \rho
	\right) }\quad \mbox{and}\quad p_{n}^{\left( \rho \right) }=\sum_{1\leq
	k\leq n}\left( b_{k,n}^{\left( \rho \right) }H_{k}^{3}+2^{-1}a_{k,n}^{\left(
	\rho \right) }H_{k}^{2}\right) ,
\end{equation*}%
where $H_{k}^{r}$ is the 
harmonic number $k$ of order $r$ as defined in~\eqref{HA},
as well as $nb_{k,n}^{\left( \rho \right) }\in \mathbb{Z}$, $n\mathcal{L}%
_{n}a_{k,n}^{\left( \rho \right) }\in \mathbb{Z}$, and taking into account
that $\mathcal{L}_{n}^{j}H_{k}^{\left( j\right) }\in \mathbb{Z}$ for $%
k=0,1,\ldots ,n$, with $j\in \mathbb{Z}^{+}$, we deduce that $nq_{n}^{\left(
	\rho \right) }\in \mathbb{Z}$ and $2n\mathcal{L}_{n}^{3}p_{n}^{\left( \rho
	\right) }\in \mathbb{Z}$. 
Thus, from Theorem \ref{Proposition_AL_RR} we have
that the characteristic equation for $\left( \text{\ref{A_Like_RR}}\right) $
is $t^{2}-34t+1=0$ and its zeros are $t_{1}=\varpi ^{4}$ and $t_{2}=\varpi
^{-4}$ respectively. 

From Poincar\'e's theorem \cite{perron,poincare}
the characteristic equation
% it 
has the behavior $q_{n}^{\left( \rho \right) }=\mathcal{O}\left( \varpi
^{4n}\right) $ and $r_{n}^{\left( \rho \right) }=\mathcal{O}\left( \varpi
^{-4n}\right) $, as $n$ goes to infinity, for the two linearly independent
solutions, respectively. Then, assuming that $\zeta \left( 3\right) =p/q$,
where $p,q\in \mathbb{Z}^{+}$, we have that $2qn\mathcal{L}%
_{n}^{3}r_{n}^{\left( \rho \right) }=2pn\mathcal{L}_{n}^{3}q_{n}^{\left(
	\rho \right) }-2qn\mathcal{L}_{n}^{3}p_{n}^{\left( \rho \right) }$, is an
integer different from zero. Therefore, as a consequence of the prime
numbers theorem we deduce that $1\leq 2qn\mathcal{L}_{n}^{3}\left\vert
r_{n}^{\left( \rho \right) }\right\vert =\mathcal{O}\left( \mathcal{L}%
_{n}^{3}\varpi ^{-4n}\right) $, which is a contradiction, and moreover $%
e^{3}\varpi ^{-4}=0,591263\ldots <1$. Clearly, the above proves Ap\'ery's
theorem.
%
%
% La ecuacin caracterstica se determina de una manera muy sencilla. Fjate bien que en la ecuacin (38) los coeficientes son polinomios del mismo grado (el grado es 9), adems los polinomios de los extremos tienen los mismo coeficientes lder, de ah basta dividir toda la ecuacin por cualquiera de estos dos coeficientes polinomiales y luego aplicar el lmite cuando n tiende al infinito y as aparece dicha ecuacin caracterstica.
%

Note that the characteristic equation $t^{2}-34t+1=0$
of \eqref{A_Like_RR} can be determined by the following steps:
The coefficients of equation \eqref{A_Like_RR} are polynomials of order the same order, namely order 9.
Moreover, the polynomials have the same leading coefficients.
Therefore it is sufficient to divide all the equation by any of these polynomial coefficients
and then apply the limit $n \rightarrow \infty$, which gives the characteristic equation.

An important consequence of the Theorem \ref{Proposition_AL_RR} is the
continued fraction representation of the number $\zeta \left( 3\right) $.
Below we present a new continued fraction expansion for $\zeta \left(
3\right) $ from our results.

\begin{theorem}
	\cite[p. 31]{Jones} Two irregular continued fractions 
	\begin{equation*}
	a_{0}+\fracd{b_{1}\mid }{\mid a_{1}}+\fracd{b_{2}\mid }{\mid a_{2}}+\fracd{
		b_{3}\mid }{\mid a_{3}}+\cdots +\fracd{b_{n}\mid }{\mid a_{n}}+\cdots ,\quad
	a_{0}^{\prime }+\fracd{b_{1}^{\prime }\mid }{\mid a_{1}^{\prime }}+\fracd{
		b_{2}^{\prime }\mid }{\mid a_{2}^{\prime }}+\fracd{b_{3}^{\prime }\mid }{\mid
		a_{3}^{\prime }}+\cdots +\fracd{b_{n}^{\prime }\mid }{\mid a_{n}^{\prime }}
	+\cdots ,
	\end{equation*}
	are equivalent if and only if there exists a sequence of non-zero $\left(
	c_{n}\right) _{n\geq 0}$ with $c_{0}=1$ such that 
	\begin{equation}
	a_{n}^{\prime }=c_{n}a_{n},\quad n=0,1,2,\ldots ,\quad b_{n}^{\prime
	}=c_{n}c_{n-1}b_{n},\quad n=1,2,\ldots  \label{NICF}
	\end{equation}
	\label{BorweinT2}
\end{theorem}
Using the previous theorems we deduce the following results.
\begin{corollary}
	Let $\rho \in \mathbb{N}$, then the following irregular continued fraction
	expansion for $\zeta \left( 3\right) $ 
	\begin{multline*}
	\zeta \left( 3\right) =\fracd{7\rho +12\mid }{\mid 6\rho +10}+\fracd{2\left(
		146\rho ^{2}+189\rho +17\right) \mid }{\mid \text{ \ \ \ \ \ \ }1654\rho
		+1981\text{ \ \ \ \ \ }} \\
	+\fracd{-16(7\rho +12)\left( 2082\rho ^{2}+1453\rho -727\right) \mid }{\mid 
		\text{ \ \ \ \ \ \ \ \ \ \ \ \ \ \ \ \ \ \ \ \ \ \ }\mathcal{Q}_{3}^{\left(
			\rho \right) }\text{\ \ \ \ \ \ \ \ \ \ \ \ \ \ \ \ \ }} \\
	+\fracd{\mathcal{P}_{4}^{\left( \rho \right) }\mid }{\mid \mathcal{Q}%
		_{4}^{\left( \rho \right) }}+\cdots +\fracd{\mathcal{P}_{n}^{\left( \rho
			\right) }\mid }{\mid \mathcal{Q}_{n}^{\left( \rho \right) }}+\cdots ,
	\end{multline*}%
	holds, where 
	\begin{multline*}
	\mathcal{P}_{n}^{\left( \rho \right) }=-(n-2)^{4}(n-1)^{4}(24n^{5}\rho
	^{2}-12n^{5}-306n^{4}\rho ^{2}+39n^{4}\rho +147n^{4} \\
	+1558n^{3}\rho ^{2}-398n^{3}\rho -684n^{3}-3959n^{2}\rho ^{2}+1532n^{2}\rho
	\\
	+1491n^{2}+5019n\rho ^{2}-2637n\rho -1470n-2538\rho ^{2}+1713\rho \\
	+473)(24n^{5}\rho ^{2}-12n^{5}-66n^{4}\rho ^{2}+39n^{4}\rho
	+27n^{4}+70n^{3}\rho ^{2} \\
	-86n^{3}\rho +12n^{3}-35n^{2}\rho ^{2}+80n^{2}\rho -45n^{2}+7n\rho ^{2} \\
	-37n\rho +30n+7\rho -7),
	\end{multline*}%
	and 
	\begin{multline*}
	\mathcal{Q}_{n}^{\left( \rho \right) }=2(n-1)(408n^{8}\rho
	^{2}-204n^{8}-3366n^{7}\rho ^{2}+663n^{7}\rho +1581n^{7} \\
	+11456n^{6}\rho ^{2}-4849n^{6}\rho -4185n^{6}-20710n^{5}\rho
	^{2}+14905n^{5}\rho \\
	+3321n^{5}+21330n^{4}\rho ^{2}-24795n^{4}\rho +4425n^{4}-12488n^{3}\rho ^{2}
	\\
	+23932n^{3}\rho -11066n^{3}+3892n^{2}\rho ^{2}-13348n^{2}\rho +8922n^{2} \\
	-504n\rho ^{2}+4008n\rho -3300n-504\rho +476).
	\end{multline*}
\end{corollary}

\begin{theorem}
	Let $\rho \in \mathbb{N}$, then the following relation 
	\begin{equation}
	\zeta \left( 3\right) =\fracd{7\rho +12}{6\rho +10}+\sum_{n\geq 1}\fracd{\Phi
		_{n}^{\left( \rho \right) }}{2n^{4}\left( n+1\right) ^{4}\Theta _{n}^{\left(
			\rho \right) }\Theta _{n+1}^{\left( \rho \right) }},  \label{eqq_2}
	\end{equation}%
	holds, where 
	\begin{equation}
	\Theta _{n}^{\left( \rho \right) }=\fracd{\rho n+1}{n}{_{5}F_{4}}\left( 
	\begin{array}{c|c}
	n+1,n,-n,-n,\rho n+2 &  \\ 
	& 1 \\ 
	1,1,1,\rho n+1 & 
	\end{array}%
	\right) ,  \label{HPhi}
	\end{equation}%
	and $\Phi _{n}^{\left( \rho \right) }$ is given in $\left( \text{\ref{Phi}}%
	\right) $.
\end{theorem}

\begin{proof}
	In fact, from $\left( \text{\ref{bes}}\right) $ and $\left( \text{\ref{AppN}}%
	\right) $ we deduce%
	\begin{equation*}
	q_{n}^{\left( \rho \right) }=\fracd{\rho n+1}{n}\sum_{0\leq k\leq n}\fracd{%
		\left( n+1\right) _{k}\left( n\right) _{k}\left( -n\right) _{k}^{2}\left(
		\rho n+2\right) _{k}}{\left( 1\right) _{k}^{2}\left( 1\right) _{k}\left(
		\rho n+1\right) _{k}}\fracd{1}{k!},
	\end{equation*}%
	which corresponds with $\left( \text{\ref{HPhi}}\right) $ according to $%
	\left( \text{\ref{rFs}}\right) $. In Addition, having in account 
	\begin{equation*}
	\fracd{p_{n}^{\left( \rho \right) }}{q_{n}^{\left( \rho \right) }}=\fracd{%
		p_{1}^{\left( \rho \right) }}{q_{1}^{\left( \rho \right) }}-\sum_{1\leq
		k\leq n-1}\left( \fracd{p_{k}^{\left( \rho \right) }}{q_{k}^{\left( \rho
			\right) }}-\fracd{p_{k+1}^{\left( \rho \right) }}{q_{k+1}^{\left( \rho
			\right) }}\right) ,
	\end{equation*}%
	and using $\left( \text{\ref{eqq_1}}\right) $ conjointly with 
	\begin{equation*}
	\zeta \left( 3\right) =\lim_{n\rightarrow \infty }\fracd{p_{n}^{\left( \rho
			\right) }}{q_{n}^{\left( \rho \right) }}=\fracd{p_{1}^{\left( \rho \right) }}{%
		q_{1}^{\left( \rho \right) }}-\sum_{n\geq 1}\left( \fracd{p_{n}^{\left( \rho
			\right) }q_{n+1}^{\left( \rho \right) }-p_{n+1}^{\left( \rho \right)
		}q_{n}^{\left( \rho \right) }}{q_{n}^{\left( \rho \right) }q_{n+1}^{\left(
			\rho \right) }}\right) ,
	\end{equation*}%
	we deduce $\left( \text{\ref{eqq_2}}\right) $. This completes the proof.
\end{proof}

\section{Convergence\label{SeccCSR}}

\subsection{Series representations}

In this paragraph several series representations of $\zeta(3)$ are
recalled. Many years after Euler's results, Chen and Srivastava (1998)
obtained several series representations for $\zeta \left( 3\right) $, which
converge faster than $\left( \text{\ref{ESR}}\right) $, including%
\begin{equation*}
\zeta \left( 3\right) =\lim_{n\rightarrow \infty }\zeta _{n}^{CS}\left(
3\right) ,
\end{equation*}%
where%
\begin{equation*}
\zeta _{n}^{CS}\left( 3\right) =-\fracd{8\pi ^{2}}{5}\sum_{k=0}^{n}\fracd{%
	\zeta \left( 2k\right) }{\left( 2k+1\right) \left( 2k+2\right) \left(
	2k+3\right) 2^{2k}}.
\end{equation*}%
Then, Srivastava (2000) \cite{Srivastava0} deduced the following result%
\begin{equation*}
\zeta \left( 3\right) =\lim_{n\rightarrow \infty }\zeta _{n}^{S}\left(
3\right) ,
\end{equation*}%
where%
\begin{equation*}
\zeta _{n}^{S}\left( 3\right) =-\fracd{6\pi ^{2}}{23}\sum_{k=0}^{n}\fracd{%
	(98k+121)\zeta \left( 2k\right) }{\left( 2k+1\right) \left( 2k+2\right)
	\left( 2k+3\right) (2k+4)(2k+5)2^{2k}}.
\end{equation*}%
In addition, Borwein et al. (2000) \cite{Borwein0} derived the following
series representation%
\begin{equation*}
\zeta \left( 3\right) =\lim_{n\rightarrow \infty }\zeta _{n}^{B}\left(
3\right) ,
\end{equation*}%
where%
\begin{equation*}
\zeta _{n}^{B}\left( 3\right) =\fracd{2\pi ^{2}}{7}\left[ \log 2-\fracd{1}{2}%
+\sum_{k=1}^{n}\fracd{\zeta \left( 2k\right) }{4^{k}\left( k+1\right) }\right]
.
\end{equation*}%
Later, Pilehrood and Pilehrood (2008) \cite{Pilehrood0} deduced the
expression%
\begin{equation*}
\zeta \left( 3\right) =\lim_{n\rightarrow \infty }\zeta _{n}^{A}\left(
3\right) ,
\end{equation*}%
where%
\begin{equation*}
\zeta _{n}^{A}\left( 3\right) =\fracd{1}{4}\sum_{k\geq 1}\left( -1\right)
^{k-1}\fracd{56k^{2}-32k+5}{k^{3}\left( 2k-1\right) ^{2}\binom{2k}{k}\binom{3k%
	}{k}},
\end{equation*}%
which is known as Amdeberhan's formula for $\zeta \left( 3\right) $, see 
\cite{Amdeberhan} for more details. Then, Pilehrood and Pilehrood (2010) 
\cite{Pilehrood1} arrived at the following expression%
\begin{equation*}
\zeta \left( 3\right) =\lim_{n\rightarrow \infty }\zeta _{n}^{PP08}\left(
3\right) ,
\end{equation*}%
where%
\begin{equation*}
\zeta _{n}^{PP08}\left( 3\right) =\sum_{k=0}^{n}\left( -1\right) ^{k}\fracd{%
	k!^{10}\left( 205k^{2}+250k+77\right) }{64\left( 2k+1\right) !^{5}},
\end{equation*}%
obtained initially by Amdeberhan and Zeilberger (1997), see \cite%
{Amdeberhan2} for more details. Analogously, Pilehrood and Pilehrood (2010) 
\cite{Pilehrood2} deduced the following formula%
\begin{equation*}
\zeta \left( 3\right) =\lim_{n\rightarrow \infty }\zeta _{n}^{PP10}\left(
3\right) ,
\end{equation*}%
where
\begin{equation*}
\zeta _{n}^{PP10}\left( 3\right) =\fracd{1}{2}\sum_{k\geq 1}\left( -1\right)
^{k-1}\fracd{205k^{2}-160k+32}{k^{5}\binom{2k}{k}^{5}}.
\end{equation*}%
More recently, Scheufens (2013) \cite{Scheufens} obtained%
\begin{equation*}
\zeta \left( 3\right) =\lim_{n\rightarrow \infty }\zeta _{n}^{Sch}\left(
3\right) ,
\end{equation*}%
where%
\begin{equation*}
\zeta _{n}^{Sch}\left( 3\right) =-\fracd{2\pi ^{2}}{7}\sum_{k=0}^{n}\fracd{%
	\zeta \left( 2k\right) }{4^{k}\left( k+1\right) \left( 2k+1\right) },
\end{equation*}
and Soria-Lorente (2014) \cite{Soria2} deduced
\begin{equation*}
\zeta \left( 3\right) =\lim_{n\rightarrow \infty }\zeta _{n}^{SL}\left(
3\right) ,
\end{equation*}%
where%
\begin{equation*}
\zeta _{n}^{SL}\left( 3\right) =\fracd{7}{6}+\sum_{k=0}^{n}\fracd{24n^{3}+30n^{2}+16n+3%
}{2n^{3}\left( n+1\right) ^{3}{\Theta }_{n}{\Theta }_{n+1}}.
\end{equation*}
Clearly, there are other series representations for $\zeta \left( 3\right)$, 
and there are ongoing investigations in this direction. 
It is important to point out that the main result obtained in this work improves
the convergence in comparison with the aforementioned results.

\subsection{Convergence rates}

If $\zeta_n(3)$ is the approximation at the $n$-th iteration and $\zeta(3)$
the exact value then the absolute error can be defined as
\begin{align}
\varepsilon_n = | \zeta_n(3) - \zeta(3) |.
\end{align}
In Figure \ref{fig:convergence}, the absolute error $\varepsilon_n$ is
visualized as a function of the index $n$ for several iteration methods.
Here, 20 iterations are realized in the index span from $n=51$ to $n=70$.

\begin{table}[htbp]
	\caption{Convergence of several iterations.}
	\label{tab:conv}
	\begin{center}
		\begin{tabular}{llllllll}
			$\zeta^{SL}$ & $\zeta^{CS}$ & $\zeta^{Sr}$ & $\zeta^B$ & $\zeta^{PP08}$ & $%
			\zeta^{A}$ & $\zeta^{PP10}$ & $\zeta^{Sch}$ \\ \hline\hline
			6,00E-159 & 8,56E-37 & 6,15E-38 & 3,48E-33 & 1,70E-77 & 6,98E-158 & 7,22E-155
			& 3,29E-35 \\ 
			5,20E-162 & 2,02E-37 & 1,43E-38 & 8,53E-34 & 6,07E-79 & 6,75E-161 & 6,98E-158
			& 7,93E-36 \\ 
			4,50E-165 & 4,79E-38 & 3,32E-39 & 2,09E-34 & 2,16E-80 & 6,54E-164 & 6,75E-161
			& 1,91E-36 \\ 
			3,90E-168 & 1,14E-38 & 7,74E-40 & 5,14E-35 & 7,72E-82 & 6,32E-167 & 6,54E-164
			& 4,61E-37 \\ 
			3,38E-171 & 2,69E-39 & 1,80E-40 & 1,26E-35 & 2,76E-83 & 6,12E-170 & 6,32E-167
			& 1,11E-37 \\ 
			2,93E-174 & 6,39E-40 & 4,21E-41 & 3,10E-36 & 9,86E-85 & 5,93E-173 & 6,12E-170
			& 2,68E-38 \\ 
			2,54E-177 & 1,52E-40 & 9,85E-42 & 7,63E-37 & 3,53E-86 & 5,74E-176 & 5,93E-173
			& 6,48E-39 \\ 
			2,20E-180 & 3,61E-41 & 2,30E-42 & 1,88E-37 & 1,26E-87 & 5,56E-179 & 5,74E-176
			& 1,57E-39 \\ 
			1,90E-183 & 8,59E-42 & 5,40E-43 & 4,61E-38 & 4,52E-89 & 5,38E-182 & 5,56E-179
			& 3,79E-40 \\ 
			1,65E-186 & 2,05E-42 & 1,26E-43 & 1,13E-38 & 1,62E-90 & 5,22E-185 & 5,38E-182
			& 9,18E-41 \\ 
			1,43E-189 & 4,87E-43 & 2,97E-44 & 2,79E-39 & 5,80E-92 & 5,05E-188 & 5,22E-185
			& 2,22E-41 \\ 
			1,24E-192 & 1,16E-43 & 6,98E-45 & 6,87E-40 & 2,08E-93 & 4,90E-191 & 5,05E-188
			& 5,38E-42 \\ 
			1,07E-195 & 2,78E-44 & 1,64E-45 & 1,69E-40 & 7,47E-95 & 4,74E-194 & 4,90E-191
			& 1,30E-42 \\ 
			9,30E-199 & 6,63E-45 & 3,86E-46 & 4,16E-41 & 2,68E-96 & 4,60E-197 & 4,74E-194
			& 3,16E-43 \\ 
			8,06E-202 & 1,58E-45 & 9,10E-47 & 1,03E-41 & 9,63E-98 & 4,46E-200 & 4,60E-197
			& 7,68E-44 \\ 
			6,98E-205 & 3,79E-46 & 2,14E-47 & 2,53E-42 & 3,46E-99 & 4,32E-203 & 4,46E-200
			& 1,86E-44 \\ 
			6,05E-208 & 9,07E-47 & 5,06E-48 & 6,23E-43 & 1,24E-100 & 4,19E-206 & 
			4,32E-203 & 4,52E-45 \\ 
			5,24E-211 & 2,17E-47 & 1,20E-48 & 1,53E-43 & 4,47E-102 & 4,06E-209 & 
			4,19E-206 & 1,10E-45 \\ 
			4,54E-214 & 5,21E-48 & 2,83E-49 & 3,78E-44 & 1,61E-103 & 3,94E-212 & 
			4,06E-209 & 2,67E-46 \\ 
			3,93E-217 & 1,25E-48 & 6,68E-50 & 9,32E-45 & 5,79E-105 & 3,82E-215 & 
			3,94E-212 & 6,49E-47%
		\end{tabular}%
	\end{center}
\end{table}

Table \ref{tab:conv} shows the convergence of several iteration methods. On
a logarithmic $y$-scale the error plot is a straight line, allowing a linear
curve fit by the exponential model 
\begin{equation}  \label{expmodel}
\varepsilon_n = qe^{\beta n}.
\end{equation}%
Taking the logarithm on both sides gives the linear model 
\begin{equation*}
\ln \varepsilon_n =\ln q+\beta n,
\end{equation*}%
where the parameters from the $\varepsilon _{n}, \; n=1,\dots ,N$ can be
calculated by solving the overdetermined system of linear equations 
\begin{equation*}
\begin{pmatrix}
1 & 1 \\ 
\vdots & \vdots \\ 
1 & N%
\end{pmatrix}%
\begin{pmatrix}
\ln q \\ 
\beta%
\end{pmatrix}%
=%
\begin{pmatrix}
\varepsilon _{1} \\ 
\vdots \\ 
\varepsilon _{N}%
\end{pmatrix}%
,
\end{equation*}%
by minimal squares. A variant of \eqref{expmodel} with arbitrary basis (e.g. 
$b=1/10$ for a decimal number system) is 
\begin{equation}  \label{expol}
\varepsilon =q(b^{n})^{r}.
\end{equation}%
The parameter $r$ in model \eqref{expol} can be deduced from model %
\eqref{expmodel} by % with $b=1/2$
$r=\beta /\ln b$. The basis $b=1/10$ gives the number of digits obtained by
one iteration; increasing the index by one corresponds to reducing the error
by the factor $(1/10)^{r}$.
In Table \ref{tab:convergence} the convergence parameters according to
several authors are compared.

\begin{figure}[tbp]
	\caption{Error reduction rates}
	\label{fig:convergence}\includegraphics[width=\textwidth]{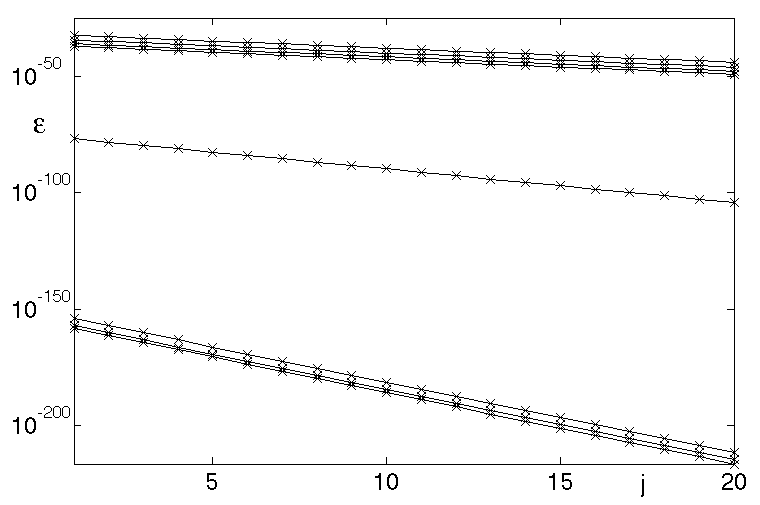}
\end{figure}

\begin{table}[htbp]
	\caption{Convergence parameters.}
	\label{tab:convergence}
	\begin{center}
		\begin{tabular}{cllllllll}
			& $\zeta^{SL}$ & $\zeta^{CS}$ & $\zeta^{Sr}$ & $\zeta^B$ & $\zeta^{PP08}$ & $%
			\zeta^{A}$ & $\zeta^{PP10}$ & $\zeta^{Sch}$ \\ \hline\hline
			$\ln q$ & -357.27 & -81.64 & -84.26 & -73.35 & -173.46 & -354.93 & -347.99 & 
			-78.00 \\ \hline
			$\beta$ & -7.05 & -1.43 & -1.45 & -1.40 & -3.33 & -6.94 & -6.94 & -1.42 \\ 
			\hline
			$q$ & 6.9e-156 & 3.5e-36 & 2.5e-37 & 1.4e-32 & 4.7e-76 & 7.2e-155 & 7.4e-152
			& 1.3e-34 \\ \hline
			$r \; (b=2)$ & 10.17 & 2.07 & 2.09 & 2.02 & 4.80 & 10.01 & 10.01 & 2.05 \\ 
			\hline
			$r \; (b=10)$ & 3.06 & 0.62 & 0.63 & 0.61 & 1.45 & 3.01 & 3.01 & 0.62%
		\end{tabular}%
	\end{center}
\end{table}
One can distinguish three groups of method that can be classified by their
convergence rate, namely Soria (2014), Amdeberhan (1996) and Pilehrood and
Pilehrood (2010) with a convergence rate of $r \approx 3$, Pilehrood and
Pilehrood (2008) with $r \approx 1.45$ and the others with $r \approx 0.6$.

The rate between two subsequent errors can be calculated from 
\begin{equation*}
\fracd{\varepsilon _{n+1}}{\varepsilon _{n}}=b^{r},
\end{equation*}%
as 
\begin{equation*}
r=\fracd{1}{\ln b} \ln \Big(\fracd{\varepsilon _{n+1}}{\varepsilon_{n}}\Big).
\end{equation*}

For the basis $b=2$ there is the general tendency that the rates decrease,
i.e. move towards the integer values (2,10), but move away from 5.

% (It would be interesting to extend the series for greater $n$) 

The methods
of Amdeberhan (1996) and Pilehrood and Pilehrood (2010) have exactly the
same error rate, which are only shifted by one index value. The reason is
that one is derived from the other such that both use the same generation
mechanism.

\subsection*{Acknowledgments}

The first author wishes to thank to Clavemat project, financed by the European
Union, and the University of Granma, where the paper was written. The second
author is supported by Conicyt (Chile) through Fondecyt project \#~1120587.

{\small \ %%%%%
%%%References should be listed in alphabetical order and include coordinates from the database Zentralblatt MATH and/or MathSciNet.
%%%%%
}

\end{document}